\documentclass[10pt]{article}

\usepackage{amsmath,amsfonts,amsthm,amssymb}
\usepackage{color}
\usepackage[]{graphics}
\usepackage{verbatim}
\usepackage{eepic,epic}
\usepackage{epsfig,subfigure,epstopdf}
\usepackage{url}
\usepackage{booktabs}
\graphicspath{{./pics/}}
\allowdisplaybreaks

\usepackage[nohead,margin=1in]{geometry}

\definecolor{forestgreen}{rgb}{0.13, 0.55, 0.13}
\definecolor{darkblue}{rgb}{0.00,0.00,0.60}

\newtheorem{theorem}{Theorem}[section]
\newtheorem{lemma}{Lemma}[section]
\newtheorem{proposition}{Proposition}[section]
\newtheorem{definition}{Definition}[section]
\newtheorem{example}{Example}[section]
\newtheorem{assumption}{Assumption}[section]
\newtheorem{remark}{Remark}[section]

\numberwithin{equation}{section}

\newcommand{\al}{\alpha}
\newcommand{\fy}{\varphi}
\renewcommand{\d}{{\rm d}}
\def\Dal{{\partial_t^\al}}

\def\Om{\Omega}
\def\II{(\Om)}

\def\Uad{\mathcal{A}}

\def\uu{\underline{u}}

\begin{document}

\title{Recovery of a Space-Time Dependent Diffusion Coefficient in Subdiffusion: Stability, Approximation and Error Analysis}

\author{Bangti Jin\thanks{Department of Computer Science, University College London, Gower Street, London, WC1E 6BT, UK.
(\texttt{b.jin@ucl.ac.uk, bangti.jin@gmail.com})}\and Zhi Zhou\thanks{Department of Applied Mathematics,
The Hong Kong Polytechnic University, Kowloon, Hong Kong (\texttt{zhizhou@polyu.edu.hk})}
}

\date{\today}

\maketitle
\begin{abstract}
{In this work, we study an inverse problem of recovering a space-time dependent diffusion coefficient in
the subdiffusion model from the distributed observation, where the mathematical model involves a
Djrbashian-Caputo fractional derivative of order $\alpha\in(0,1)$ in time. The main technical
challenges of both theoretical and numerical analysis lie in the limited smoothing properties due to
the fractional differential operator and high degree of nonlinearity of the forward map from the
unknown diffusion coefficient to the distributed observation. We establish two conditional stability results
using a novel test function, which leads to a stability bound in $L^2(0,T;L^2(\Omega))$ under a suitable
positivity condition. {The positivity condition is verified for a large class of problem data.}
Numerically, we develop a rigorous procedure for recovering the diffusion coefficient based on a
regularized least-squares formulation, which is then discretized by the standard Galerkin method with
continuous piecewise linear elements in space and backward Euler convolution quadrature in time.
We provide a complete error analysis of the fully discrete formulation, by combining several new error
estimates for the direct problem (optimal in terms of data regularity), a discrete version of fractional
maximal $L^p$ regularity, and a nonstandard energy argument. Under the positivity condition, we obtain a
standard $\ell^2(L^2\II)$ error estimate consistent with the conditional stability. Further, we
illustrate the analysis with some numerical examples.}
\end{abstract}

\section{Introduction}\label{sec:intro}
This work is concerned with a parameter identification problem for the subdiffusion model with a space-time-dependent
diffusion coefficient and its rigorous numerical analysis. Let $\Omega\subset\mathbb{R}^d$ ($d=1,2,3$) be a convex
polyhedral domain with a boundary $\partial\Omega$. Fix $T>0$ the final time. Consider the following initial-boundary
value problem for the function $u$:
\begin{equation}\label{eqn:fde}
    \left\{\begin{aligned} \Dal  u -\nabla\cdot(q\nabla u)  &= f,&&  \mbox{in }\Omega\times(0,T],\\
    u(\cdot,0) &= u_0, && \mbox{in }\Omega,\\
u &=0,&& \mbox{on }\partial\Omega\times(0,T],
   \end{aligned} \right.
\end{equation}
where the functions $f$ and $u_0$ are the given source and initial condition, respectively, and
the diffusion coefficient $q$ is assumed to be space-time dependent.
The notation $\Dal u$ denotes the Djrbashian-Caputo fractional derivative in time $ t$ of order
$\alpha\in(0,1)$, defined by (see e.g., \cite[p. 92]{KilbasSrivastavaTrujillo:2006} and \cite[Section 2.3]{Jin:2021book})
\begin{equation*}
  \Dal u (t) = \frac{1}{\Gamma(1-\alpha)}\int_0^t (t-s)^{-\alpha}  u'(s)\ {\rm d}s,
\end{equation*}
where $\Gamma(z)=\int_0^\infty s^{z-1}e^{-s}\d s$ (for $\Re(z)>0$) denotes Euler's Gamma function.
The fractional derivative $\partial_t^\alpha u$ recovers the usual first order derivative $u'$
as the order $\alpha\to1^-$ for sufficiently smooth functions $u$. Thus the model \eqref{eqn:fde}
is a fractional analogue of the classical diffusion model. Below we use the notation $u(q)$ to
explicitly indicate the dependence of the solution $u$ on $q$.
The model \eqref{eqn:fde} has received enormous attention in recent years in physics, engineering,
biology and finance, due to their excellent capability for describing anomalously slow diffusion
processes, also known as subdiffusion, which displays local motion occasionally interrupted by
long sojourns and trapping effects. These transport processes are characterized by a sublinear
growth of the mean squared displacement of the particle with the time, as opposed to linear
growth for Brownian motion. The model \eqref{eqn:fde} has found many successful practical
applications, e.g., diffusion in fractal domains (see e.g., \cite{Nigmatulli:1986}),
transport column experiments (see e.g., \cite{HatanoHatano:1998}), and subsurface flows (see e.g., \cite{AdamsGelhar:1992});
see \cite{MetzlerKlafter:2000,MetzlerJeon:2014} for physical modeling and a long list
of applications.

This work is concerned with recovering the space-time dependent diffusion
coefficient $q^\dag(x,t)$ in the model \eqref{eqn:fde} from the (noisy) distributed observation
\begin{equation}\label{eqn:observ}
   z^\delta(x,t) = u(q^\dag)(x,t) + \xi(x,t),\quad (x,t)\in\Omega\times[0,T],
\end{equation}
where $\xi(x,t)$ denotes the pointwise additive noise, with a noise level
$\delta=\|u(q^\dag)-z^\delta\|_{L^2(0,T;L^2(\Omega))}$.
The exact diffusion coefficient $q^\dag$ is sought in the following admissible set
\begin{equation}\label{eqn:Admset}
    \Uad=\{ q\in L^\infty((0,T)\times\Omega):~~c_0\le q \le c_1, ~~\text{a.e. in}~~ \Omega\times(0,T)\},
\end{equation}
with $0<c_0<c_1<\infty$.
The inverse problem is a fractional analogue of the inverse conductivity problem for standard parabolic problems, which
has been extensively studied both numerically and theoretically (see
\cite{Isakov:2006,BanksKunisch:1989,Chavent:2009} and the references therein).

The inverse problem of recovering a space-time dependent diffusion coefficient $q(x,t)$ is formally determined for
uniqueness / identifiability. Despite its obvious practical relevance (see \cite{FaLenzi:2005,GarraOrsingher:2015}), to
the best of our knowledge, it has not been studied so far. In this work, we contribute to its mathematical and numerical analysis.
First, we establish two conditional stability results in Theorems \ref{thm:stab1} and \ref{thm:stab2}. These
estimates allow deriving the standard $L^2(0,T;L^2(\Omega))$ stability under a suitable positivity condition that can
be verified for a class of problem data.
These results are proved using a novel test function (inspired by \cite{Bonito:2017}) together with
refined regularity results for the direct problem. Second, we develop a numerical procedure for
recovering a space-time dependent diffusion coefficient, using an output least-squares formulation with a
space-time $H^1$-seminorm penalty at both continuous and discrete levels, and discuss their well-posedness.
Third, we derive a weighted $L^2(\Omega)$ error estimate for discrete approximations under a mild
regularity assumption on the exact diffusion coefficient $q^\dag(x,t)$; see Theorem \ref{thm:error-q} for the
precise statement. The analysis is inspired by the conditional stability analysis, assisted with several new
nonsmooth data error estimates in the appendix. Further, we provide several numerical
experiments to complement the theoretical analysis. Due to the nonlocality of the operator
$\partial_t^\alpha$, the solution operator has only limited smoothing properties (see
\cite{KubicaYamamoto:2020,Jin:2021book} for the solution theory) and the forward  map is highly nonlinear,
which represent the main technical challenges in the analysis. To overcome these challenges, we employ the following
powerful analytical tools for evolution problems, e.g., maximal $L^p$ regularity, nonsmooth data estimates and novel test function $\fy$.

Now we briefly review existing works. Inverse problems for anomalous diffusion
has attracted much interest, and there is a vast literature  (see, e.g., the reviews
\cite{JinRundell:2015} and \cite{LiYamamoto:2019review}). A number of works studied recovering a spatially dependent
diffusion coefficient (see e.g., \cite{ChengYamamoto:2009,LiGuJia:2012,LiYamamoto:2013,Zhang:2016,KianOksanenYamamoto:2018}).
\cite{ChengYamamoto:2009} proved the unique recovery of both diffusion coefficient and fractional order
$\alpha$ from the lateral Cauchy data for the model \eqref{eqn:fde} with a Dirac source in the one-dimensional case
using Laplace transform and Sturm-Liouville theory. See also \cite{KianOksanenYamamoto:2018} for recovering two
coefficients from the Dirichlet-to-Neumann map. \cite{Zhang:2016} proved the unique recovery of $q(t)$ from
lateral Cauchy data; see also \cite{LopushanskyiLopushanska:2014}. Nonetheless, there seems still no known stability
result for the inverse problem, and Theorems \ref{thm:stab1} and \ref{thm:stab2} are first known stability results
for the concerned inverse problem. We also refer readers to \cite{KaltenbacherRundell:2019, ZhangZhou:2017} for
the closely related inverse potential problem, and \cite{KR:2019} for recovering a nonlinear reaction term in a
fractional reaction-diffusion equation. \cite{LiGuJia:2012, LiYamamoto:2013} discussed the numerical recovery of
the diffusion coefficient $q(x)$ and fractional order $\alpha$, but the numerical discretization was not analyzed.
See also \cite{WeiLi:2018} for further numerical results on recovering the diffusion coefficient from boundary
data in the one-dimensional case, using a space-time variational formulation,
which allows only a zero initial condition. In summary,
existing works have not studied discretization schemes in a proper functional analytic setting, and this
represents one gap that this work aims to fill in. Previously \cite{JinZhou:2021sicon} analyzed the inverse
problem of recovering a spatially-dependent diffusion coefficient $q(x)$ from distributed observation, and provided a
convergence (rate) analysis of the discrete approximation; see also \cite{WangZou:2010,JinZhou:2021sinum} for the
standard parabolic case. This work substantially extends \cite{JinZhou:2021sicon} in the following
aspects: (1) we provide novel conditional stability estimates; (2) the error analysis covers the one- to three-dimensional case,
whereas that in \cite{JinZhou:2021sicon} is restricted to one- and two-dimensional cases, due to certain regularity
lifting. This restriction is overcome
by using maximal $L^p$ regularity for the direct problem and maximal $\ell^p$ regularity for the time-stepping scheme.
(3) the presence of time-dependence of the diffusion coefficient $q$
poses significant challenge in the analysis and numerics, for which we shall develop the requisite analytic tools below.
Thus the extension requires new technical developments that are still unavailable in the existing literature.

The rest of the paper is organized as follows. In Section \ref{sec:cont}, we give preliminary well-posedness results for
the direct problem, especially regularity. In Section
\ref{sec:stab} we present two conditional stability results. Then in Section \ref{sec:reg}, we describe the regularized
formulation, and its numerical discretization for the recovery of $q(x,t)$. Next, in Section \ref{sec:err}, we present
an error analysis of the fully discrete scheme. Finally, in Section \ref{sec:numer}, we present illustrative numerical
results to complement the analysis. Throughout, the notation $c$, with or without a subscript, denotes a generic
constant which may change at each occurrence, but it is always independent of the following parameters: regularization
parameter $\gamma$, mesh size $h$, time stepsize $\tau$ and noise level $\delta$. For a bivariate function $f(x,t)$,
we often write $f(t)=f(\cdot,t)$ as a vector valued function.

\section{Well-posedness of the forward problem}\label{sec:cont}

First we describe some regularity results for the direct problem \eqref{eqn:fde}. Since it involves the time-dependent
coefficient $q(x,t)$, its well-posedness analysis requires extra care \cite[Chapter 4]{KubicaYamamoto:2020}
\cite[Section 6.3]{Jin:2021book}. Below we revisit the regularity results, which is needed for the analysis
 in Sections \ref{sec:stab} and \ref{sec:reg}.

First we describe the functional analytic setting. For any $r\geq1$, we denote by $r^*\geq1$ its
conjugate exponent, i.e., $\frac1r+\frac1{r^*}=1$. For any $k\geq 0$ and $p\geq1$,
the space $W^{k,p}(\Omega)$ denotes the standard Sobolev space of the $k$th order, and we
write $H^k(\Omega)$, when $p=2$. The dual spaces of $W_0^{1,p}(\Omega)$
and $H_0^1(\Omega)$ are denoted by $(W_0^{1,p})' = W^{-1,p^*}(\Omega)$ and $(H_0^1(\Omega))'=H^{-1}(\Omega)$,
respectively. The notation $(\cdot,\cdot)$ denotes the $L^2(\Omega)$ inner product and also the duality
between $W^{1,p}_0\II$ and $W^{-1,p^*}\II$. For a UMD space $X$ (see \cite[Section 4.2.c]{Hytonen:2016} for the
definition and examples of UMD spaces, which include Sobolev spaces $W^{s,p}(\Omega)$ with $1<p<\infty$
and $s\ge 0$), we denote by $W^{s,p}(0,T;X)$ the space of vector-valued functions $v:(0,T)\rightarrow X$,
with the norm $\|\cdot\|_{W^{s,p}(0,T;X)}$ defined by complex interpolation:
\begin{align*}
\|v\|_{W^{s,p}(0,T;X)}&:= \inf_{\widetilde v}\|\widetilde v\|_{W^{s,p}({\mathbb R};X)}
:= \inf_{\widetilde v} \|  \mathcal{F}^{-1}[ (1+|\xi|^2)^{\frac{s}{2}} \mathcal{F}[\widetilde v](\xi) ]\|_{L^p(\mathbb{R};X)},
\end{align*}
where the infimum is taken over all possible functions $\widetilde v$ that extend $v$ from $(0,T)$ to ${\mathbb R}$,
and $\mathcal{F}$ denotes the Fourier transform.
For any $r\in(1,\infty)$, we define a time-dependent elliptic operator $A(t)\equiv A(t;q):
W_0^{1,r}\II \rightarrow (W_0^{1,r^*}(\Omega))' = W^{-1,r}\II$ by
\begin{equation}\label{eqn:op-A}
  (A(t) u, \phi) = (q(t) \nabla u, \nabla \phi), \quad \forall u \in W_0^{1,r}\II, \phi \in W_0^{1,r^*}\II.
\end{equation}
The dependence of $A(t)$ on $q$ will be suppressed whenever there is no confusion.
Also we denote by $A=-\Delta$, the negative Dirichlet Laplacian,
i.e., $q(x,t)\equiv1$.  Throughout, for the convex polygonal domain $\Omega\subset\mathbb{R}^d$,
we assume that there exists $r>\min(d,2)$ such that the full second-order elliptic regularity pickup in $L^p\II$ holds.

Now we can introduce the concept of a weak solution.
\begin{definition}\label{def:weak}
For $r\geq2$ and $p>\frac2\alpha$, a function $u\in L^p(0,T; W_0^{1,r}\II)\cap C([0,T];L^r\II)$
is said to be a weak solution to problem \eqref{eqn:fde} if  $\partial_t^\alpha u \in L^p(0,T;
W^{-1,r}\II)$ and it satisfies
\begin{equation}\label{eqn:weak}
(\partial_t^\alpha u(t), \phi) + (q(t)\nabla u(t),\nabla \phi) = (f(t),\phi), \quad \forall \phi \in W_0^{1,r^*}(\Omega), t\in(0,T],
\end{equation}
with the initial condition $u(0) = u_0$ in $L^r(\Omega)$.
\end{definition}

To study the well-posedness of problem \eqref{eqn:fde}, we make the following assumption.
\begin{assumption}\label{ass:data-1}
The diffusion coefficient $q$, initial data $u_0$ and source $f$ satisfy
\begin{itemize}
\item[{\rm (i)}] $q\in \Uad$, $q \in C^1([0,T];C(\overline \Omega))\cap C([0,T];C^\mu(\overline\Omega))$, with some $\mu\in(0,1)$;
\item[{\rm (ii)}] $u_0\in W_0^{1,r}(\Omega)$ and  $f\in L^p(0,T; W^{-1,r}\II)$ with some $p\in(\frac2\alpha,\infty)$ and $r \in [2, \infty)$.
\end{itemize}
\end{assumption}

Now we recall two preliminary results. The first is a perturbation estimate.
\begin{lemma}\label{lem:perturb}
If $q\in \Uad$ and $|q_t(x,t)| \le M$, then the operator $A(t)\equiv A(t;q)$ satisfies
\begin{equation*}
\| (A(t) - A(s) ) u \|_{W^{-1,r}(\Omega)} \leq c|t-s|\|\nabla u\|_{L^r(\Omega)}.
\end{equation*}
\end{lemma}
\begin{proof}
It follows directly from the definition and the condition $|\partial_tq|\leq M$ that
\begin{align*}
& \| (A(t) - A(s) ) u \|_{W^{-1,r}\II}  = \sup_{v\in W_0^{1,r^*}\II} \frac{( (A(t) - A(s) ) u, v)}{\|  \nabla v \|_{L^{r^*}\II}}\\
=&  \sup_{v\in W_0^{1,r^*}\II} \frac{( (q(t) - q(s) )\nabla  u, \nabla v)}{\|  \nabla v \|_{L^{r^*}\II}} \le c|t-s| \| \nabla u  \|_{L^r\II}
\end{align*}
This shows the desired estimate.
\end{proof}

The second is the maximal $L^p$ regularity for the model \eqref{eqn:fde} with
a stationary diffusion coefficient.
\begin{lemma}\label{lem:max-weak}
If $q$ is independent of $t$ and $q\in C^{\mu}(\overline\Omega)\cap \mathcal{A}$ with
$\mu\in(0,1)$, then for $u_0=0$ and $f\in L^p(0,T; W^{-1,r}\II)$ with $r\ge2$ and
$p>\frac2\alpha$, problem \eqref{eqn:fde} admits a unique weak solution $u$ and
\begin{equation*}
\|  \partial_t^\alpha u \|_{L^p(0,T; W^{-1,r}(\Omega))} + \|  \nabla u  \|_{L^p(0,T; L^r(\Omega))} \le c  \|  f  \|_{L^p(0,T; W^{-1,r}(\Omega))}.
\end{equation*}
\end{lemma}
\begin{proof}
For $r=2$ and $p=2$, the estimate can be found in \cite[Exercise 6.5]{Jin:2021book},
and the case $p\in (1,\infty)$ follows similarly. Thus $u\in W^{\alpha,p}(0,T;H^{-1}\II)\cap L^p(0,T;
H_0^1\II)$, and since $p>\frac2\alpha$, the interpolation between $W^{\alpha,p}(0,T;H^{-1}\II)$
and $L^p(0,T; H_0^1\II)$ \cite[Theorem 5.2]{Amann:2000} and Sobolev embedding theorem \cite{AdamsFournier:2003} imply $u\in
C([0,T]; L^2\II)$. For $r>2$, the condition $q\in C^{\mu}(\overline\Omega)$ implies that the
operator $A$ is $R$-sectorial on $W^{-1,r}\II$ with an angle $\frac\pi2$ \cite[Lemma 8.5]{AkrivisLiLubich:2017}.
Then the maximal $L^p$ regularity follows as \cite[Theorem 6.11]{Jin:2021book}.
\end{proof}

Now we can state the existence and uniqueness of a weak solution to problem \eqref{eqn:fde} in
the sense of Definition \ref{def:weak}. See the appendix for the proof.
\begin{theorem}\label{thm:weak-sol}
Let Assumption \ref{ass:data-1} be fulfilled. Then
problem \eqref{eqn:fde} admits a unique weak solution in the sense of Definition \ref{def:weak}.
Further, if $r>d$ and $p>\frac{2r}{\alpha(r-d)}$, then $u\in L^\infty((0,T)\times \Omega)$.
\end{theorem}

Next, we derive several improved regularity estimates.
\begin{assumption}\label{ass:data-2}
The diffusion coefficient $q$, initial data $u_0$ and source $f$ satisfy the following assumptions.
\begin{itemize}
\item[{\rm(i)}] $q^\dag \in \Uad$ and the following condition holds
\begin{align}
 |\partial_t q(x,t)|+|\nabla_x q(x,t)|+|\nabla_x\partial_t q(x,t)|\le M, \quad \forall (x,t)\in \Omega\times(0,T]. \label{Cond-q}
\end{align}
\item[{\rm(ii)}] $u_0\in W^{2,r}(\Omega)\cap H_0^1(\Omega)$, with some $r>\max(2,d)$, and $f\in L^\infty((0,T)\times\Omega)\cap C^1([0,T];L^2(\Omega))$.
\end{itemize}
\end{assumption}

Under Assumption \ref{ass:data-2}, the operator $A(t) := A(t;q)$ satisfies that
for $\beta\in[0,1]$ and $t,s\in[0,T]$ \cite[Lemma 6.5]{Jin:2021book}
\begin{equation}\label{eqn:perturb}
 \| A(t)^{\beta}(I-A(t)^{-1}A(s)) \phi  \|_{L^2\II} \le c |t-s| \| A^{\beta} \phi \|_{L^2\II},\quad \forall \phi\in D(A^{\beta}).
\end{equation}

The next result gives an improved regularity estimate.
\begin{proposition}\label{prop:reg-1}
Under Assumption \ref{ass:data-2}, problem \eqref{eqn:fde} has a unique solution $u\in L^p(0,T; W^{2,r}(\Omega))
\cap W^{\alpha,p}(0,T;L^r\II)$ for any $p\in(\frac2\alpha,\infty)$.
\end{proposition}
\begin{proof}
By Theorem  \ref{thm:weak-sol}, it suffices to show the regularity.
Let $w=u-u_0$, which satisfies
\begin{equation*}
    \Dal  w -A(t) w  = f + A(t)u_0 ,\quad \forall t\in (0,T],\quad \mbox{with }w(0)=0.
\end{equation*}
Since $q\in \Uad$ and satisfies condition \eqref{Cond-q}, and $u_0 \in  W^{2,r}\II \cap H_0^1\II$,
$f+A(t)u_0$ belongs to $L^p(0,T;L^r(\Omega))$. The standard maximal $L^p$ regularity and
the argument in \cite[Theorem 6.14]{Jin:2021book} imply
$$
w \in L^p(0,T; W^{2,r}(\Omega))\quad\mbox{and}\quad\partial_t^\alpha w \in L^p(0,T; L^r\II).
$$
This and $w(0)=0$ imply $w \in W^{\alpha,p}(0,T;L^r\II)$ \cite[Lemma 2.4]{JKZ:2021}. So $u = w+u_0
\in  L^p(0,T; W^{2,r}(\Omega)) \cap W^{\alpha,p}(0,T;L^r\II)$.
\end{proof}\vskip5pt

{By Proposition \ref{prop:reg-1}, interpolation theorem \cite[Theorem 5.2]{Amann:2000}
and Sobolev embedding theorem, we deduce that for any $\theta<\frac12-\frac{d}{2r}$ and $p>
\frac1{\alpha\theta}$, there holds
\begin{equation}\label{reg-fde-inf}
  u \in W^{\alpha\theta,p}(0,T; W^{2(1-\theta),r}(\Omega)) \hookrightarrow C([0,T]; W^{1,\infty}(\Omega)).
\end{equation}
Further, by \cite[Theorems 6.15 and 6.16]{Jin:2021book} and the full elliptic regularity pickup, there holds
\begin{equation}\label{reg-fde-2}
 \|u(t)\|_{H^2(\Omega)} + \| \partial_t^\alpha u(t) \|_{L^2(\Omega)} + t^{1-\alpha}\| \partial_{t}u(t) \|_{L^2(\Omega)}
  + t \| \partial_{t}u(t) \|_{H^2(\Omega)} \le c,\quad \forall t\in (0,T].
\end{equation}}

The next result gives a weighted bound on $u'(t)$. This estimate will play a role in the conditional stability
analysis in Section \ref{sec:stab} and the error analysis in Section \ref{sec:err}.
\begin{proposition}\label{prop:reg-2}
Let Assumption \ref{ass:data-1} with $r=2$ and \eqref{Cond-q} hold. Then for
$f\in C([0,T]; H^{-1}(\Omega))$ with $\int_0^t  (t-s)^{\frac\alpha2-1}\| f'(s) \|_{H^{-1}(\Omega)}\,\d s< c$,
there holds for any small $\epsilon>0$
$$ \Big\| \int_s^t(\xi-s)^{-\alpha} u'(\xi) \,\d \xi \Big\|_{L^2\II} \le c_\epsilon s^{-\frac\alpha2-\epsilon}.$$
\end{proposition}
\begin{proof}
Under the given data regularity assumption,  we claim
\begin{align}\label{eqn:der-bound}
  \| u '(t) \|_{L^2(\Omega)} \le c t^{\frac\alpha2-1} ,\quad \forall t\in(0,T].
\end{align}
Then for any $\epsilon>0$, the desired assertion follows directly as
\begin{align*}
 &\quad\Big\| \int_s^t(\xi-s)^{-\alpha} u'(\xi) \,\d \xi \Big\|_{L^2\II} \le  \int_s^t(\xi-s)^{-\alpha} \| u'(\xi)\|_{L^2\II} \,\d \xi \\
 &\le c \int_s^t(\xi-s)^{-\alpha}  \xi^{\frac\alpha2-1} \,\d \xi
 \le  c s^{-\frac\alpha2-\epsilon}\int_s^t(\xi-s)^{-\alpha}  \xi^{\epsilon+\alpha-1} \,\d \xi\\
&\le c s^{-\frac\alpha2-\epsilon} \int_s^t(\xi-s)^{\epsilon-1}   \,\d \xi   \le c\epsilon^{-1} s^{-\frac\alpha2-\epsilon}.
\end{align*}
It remains to prove the claim \eqref{eqn:der-bound}. We fix $t_*\in (0,T]$, 
and represent the solution $u$ by (with $A_*\equiv A(t_*)$)
\begin{align}\label{eqn:sol-op-rep}
 u(t) = F_*(t) u_0 + \!\int_0^t E_*(s) f(t-s) \d s + \!\int_0^t  E_*(t-s)  (A(t_*)- A(s))u(s) \d s,
\end{align}
where $F_*(t) = \frac{1}{2\pi\rm i}\int_{\Gamma_{\theta,\delta}}e^{zt}z^{\alpha-1}(A_*+z^\alpha)^{-1}\d z$
and $E_*(t) = \frac{1}{2\pi\rm i}\int_{\Gamma_{\theta,\delta}}e^{zt}(A_*+z^\alpha)^{-1}\d z$ denote the
solution operators for the initial data and source, respectively, with the contour
$\Gamma_{\theta,\delta} =\{z=re^{\pm {\rm i}\theta}, r\ge \delta   \}
\cup\{ z= \delta e^{{\rm i}\varphi}: |\varphi|\leq \theta\}$, with $\theta\in(\frac\pi2,\pi)$. The following
smoothing properties hold \cite[Theorem 6.4]{Jin:2021book}:
\begin{equation*}
\|  F_*'(t) v\|_{L^2\II} \le c t^{\frac{\alpha}{2}-1}\|   \nabla v \|_{L^2\II}~~\text{and}~~  \| E_*(t) v  \|_{L^2\II} \le c t^{\beta\alpha-1} \| A^{\beta-1}v \|_{L^2\II}.
\end{equation*}
Meanwhile, it follows from the representation \eqref{eqn:sol-op-rep} of $u$ that
\begin{equation*}
\begin{aligned}
  u'(t) &= F_*'(t) u_0 + E_*(t)f(0) + \int_0^t E_*(s) \frac{\d}{\d t}f(t-s) \,\d s + E_*(t)  (A_*- A(0))u_0\\
     & \quad + \int_0^t E_*(s)((A_*-A(t-s))\frac{\d}{\d t}u(t-s) + [\frac{\d}{\d t}A(t-s)]u(t-s))\,\d s.
\end{aligned}
\end{equation*}
Taking $L^2\II$ norm on both sides,  setting $t$ to $t_*$ and the perturbation estimate \eqref{eqn:perturb} lead to
\begin{align*}
 \| u'(t_*) \|_{L^2\II} &= c t_*^{\frac\alpha2-1}(\| {\nabla u_0}\|_{L^2\II} +  \|f(0)\|_{H^{-1}\II}) + c\int_0^{t_*}  \| u'(s) \|_{L^2\II} \,\d s\\
   &\quad + c\int_0^{t_*} (t_*-s)^{\frac\alpha2-1} \| f'(s)\|_{H^{-1}\II} \,\d s  + c \int_0^{t_*}  ({t_*}-s)^{\frac\alpha2-1}\|  \nabla u(s)\|_{L^2\II}\,\d s.
\end{align*}
Given the regularity of $u_0$ and $f$, we have $u\in C([0,T]; L^2\II) \cap L^p (0,T; H_0^1\II)$ for
any $p\in(\frac2\alpha,\infty)$, cf. Theorem \ref{thm:weak-sol}, which implies
$  \int_0^t  (t-s)^{\frac\alpha2-1}\|  \nabla u(s)\|_{L^2\II}\,\d s < c$, for $ t\in(0,T].$ Thus, we obtain
\begin{align*}
  \| u'(t_*) \|_{L^2\II} &\le c t_*^{\frac\alpha2-1}   + \int_0^{t_*}  \| u'(s) \|_{L^2\II} \,\d s,  \quad \forall t_*\in(0,T].
\end{align*}
Then the standard Gronwall's inequality implies the desired claim \eqref{eqn:der-bound}, completing the proof of the proposition.
\end{proof}

\section{Conditional stability}\label{sec:stab}

In this section, we establish two novel conditional stability results for the concerned inverse problem,
which serve as a benchmark for the convergence rates of the numerical approximations. To the best of our
knowledge, they represent the first stability results for the concerned inverse problem, and are of
independent interest. We introduce a positivity condition, with $\mathrm{dist}(x,\partial\Omega) =
\inf_{x'\in\partial\Omega}|x-x'|$, which will be verified for a class of problem data.
\begin{definition}\label{def:condP1}
The solution $u$ to problem \eqref{eqn:fde} is said to satisfy the $\beta$-positivity condition
with $\beta\ge0$, if for any $(x,t)\in \Omega\times (0,T)$
\begin{equation*}
 q(x,t) |\nabla u(q)(x, t)|^2 +(f(x,t) -\partial_t^\alpha u(q)(x, t))u(q)(x, t) \ge c \,\mathrm{dist}(x, \partial\Omega)^\beta.
\end{equation*}
\end{definition}

Now we state the first conditional stability estimate for the inverse problem.
\begin{theorem}\label{thm:stab1}
Let $u_0$, $f$ and $q_i$, $i=1,2$, satisfy Assumption \ref{ass:data-1} with
$r>d$ and $p>\frac{2r}{\alpha(r-d)}$, and $ \| \nabla q_i \|_{L^2(0,T;L^2\II)}\le c$, $i=1,2.$
Let $u_i\equiv u(q_i)$ be the solution to problem \eqref{eqn:fde}.
Then there holds 
\begin{align*}
&\quad \int_0^T\int_\Omega \Big(\frac{q_1-q_2}{q_1}\Big)^2  \big(q_1 |\nabla u_1 |^2 +( f -\partial_t^\alpha u_1)u_1\big)\,\d x \d t\\
&\le c\big(\|  \nabla (u_1 - u_2)\|_{L^2(0,T;L^2\II)} + \|   {\partial_t^\alpha} (u_1-u_2)\|_{L^2(0,T; H^{-1}\II)} \big).
\end{align*}
Further, if the solution $u_1$ to problem \eqref{eqn:fde} satisfies the $\beta$-positivity condition, then
\begin{align*}
\|q_1-q_2\|_{L^2(0,T; L^2\II)}
\le &c\big(\|  \nabla (u_1 - u_2) \|_{L^2(0,T; L^2\II)} + \|   {\partial_t^\alpha} (u_1-u_2) \|_{L^2(0,T; H^{-1}\II)} \big)^{\frac{1}{2(1+\beta)}}.
\end{align*}
\end{theorem}
\begin{proof}
Assumption \ref{ass:data-1} and Theorem \ref{thm:weak-sol} imply that problem \eqref{eqn:fde} has a
weak solution  $u_i \in L^\infty(\Omega\times(0,T)) \cap L^2(0,T; H_0^1\II)$.
This and the assumption $q_i \in \mathcal{A} \cap L^2(0,T; H_0^1\II$ imply
$\varphi = \frac{q_1-q_2}{q_1} u_1 \in L^2(0,T; H_0^1\II)$. Indeed, the choice $\varphi$ gives
$\nabla \varphi = \frac{q_1-q_2}{q_1}\nabla u_1 + \frac{q_1\nabla (q_1-q_2)-(q_1-q_2)\nabla q_1}{q_1^2}u_1$.
Then by the triangle inequality, Assumption \ref{ass:data-1} and the condition $\| \nabla q_i\|_{L^2(0,T;L^2\II)}\le c$, $i=1,2$, we have
\begin{equation*}
\begin{split}
  \| \nabla \fy \|_{L^2(0,T;L^2\II)}^2 &\le c \int_0^T \|  \frac{q_1 - q_2}{q_1} (t)\|_{L^\infty\II}^2 \|  \nabla u_1(t) \|_{L^2\II} ^2
  \d t\\
  &\quad  + c \int_0^T\| \frac{q_1\nabla (q_1-q_2) - (q_1-q_2)\nabla q_1}{q_1^2} (t) \|_{L^2\II}^2 \|   u_1(t) \|_{L^\infty \II}^2 \,\d t\\
  &\le   c \int_0^T   \|  \nabla u_1 (t) \|_{L^2\II} ^2 \,\d t +  c \|u_1\|_{L^\infty((0,T)\times \Omega)}^2 .
\end{split}
\end{equation*}
Then by the regularity $u_1 \in L^2(0,T;H_0^1\II) \cap L^\infty((0,T)\times \Omega)$ from
Theorem \ref{thm:weak-sol}, we deduce $ \| \nabla \fy \|_{L^2(0,T;L^2\II)} \le c.$ By taking $\varphi
(t)= (\frac{q_1-q_2}{q_1}u_1)(t)$ and integration by parts,
\begin{align*}
  2((q_1-q_2)\nabla u_1,\nabla\varphi)(t)&  = (\tfrac{q_1-q_2}{q_1}q_1\nabla u_1,\nabla\varphi)(t) + ((q_1-q_2)\nabla u_1,\nabla\varphi)(t)\\
   & =-(q_1\nabla (\tfrac{q_1-q_2}{q_1})\cdot\nabla u_1,\varphi)(t) - ((\tfrac{q_1-q_2}{q_1})\nabla\cdot(q_1\nabla u_1),\varphi)(t)\\
    &\quad + ((q_1-q_2)\nabla u_1,\nabla\varphi)(t).
\end{align*}
Using the identity {$-\nabla \cdot(q_1\nabla u_1)=f-\partial_t^\alpha u_1$} and inserting the choice $\varphi=\frac{q_1-q_2}{q_1}u_1$ in the third term gives
\begin{align*}
  2((q_1-q_2)\nabla u_1,\nabla\varphi)(t)
   & =-(q_1\nabla (\tfrac{q_1-q_2}{q_1})\cdot\nabla u_1,\tfrac{q_1-q_2}{q_1}u_1)(t) + ((\tfrac{q_1-q_2}{q_1})^2{ (f-\partial_t^\alpha u_1)},u_1)(t)\\
    &\quad + ((q_1-q_2)\nabla u_1,\nabla(\tfrac{q_1-q_2}{q_1})u_1 + \tfrac{q_1-q_2}{q_1}\nabla u_1)(t).
\end{align*}
Collecting the terms gives the following crucial identity
\begin{align}\label{eqn:crucial-identity}
((q_1-q_2)\nabla u_1,\nabla\varphi)(t)&= \frac{1}{2}\int_\Omega \Big(\frac{q_1-q_2}{q_1}(t)\Big)^2
\big( q_1(t)|\nabla u_1(t)  |^2
+(f(t)-\partial_t^\alpha u_1(t))u_1(t)\big)\,\d x.
\end{align}
Meanwhile, the variational formulation \eqref{eqn:weak} for $u_i$ implies that for any fixed $t\in(0,T)$,
\begin{align}
 &\quad ((q_1-q_2)\nabla u_1,\nabla\varphi)(t)=  (q_1\nabla u_1,\nabla\varphi)(t) -(q_2\nabla u_1,\nabla\varphi)(t) \nonumber\\
 & = (f, \varphi)(t) - (\partial_t^\alpha u_1 ,\varphi)(t)  -(q_2\nabla u_1,\nabla\varphi)(t) \nonumber\\
 & = (\partial_t^\alpha u_2, \varphi )(t) + (q_2 \nabla u_2,\nabla\varphi)(t)-(\partial_t^\alpha u_1 ,\varphi)(t)-(q_2\nabla u_1,\nabla\varphi)(t)\nonumber\\
  &= {-(q_2\nabla(u_1-u_2),\nabla \varphi)(t)} - {(\partial_t^\alpha(u_1-u_2), \varphi)(t)}.\label{eqn:identity-1}
\end{align}
By the Cauchy--Schwarz inequality, we have
\begin{equation*}
\begin{split}
 &((q_1-q_2)\nabla u_1,\nabla\varphi)(t)\\
 \le& c \big(\|  \nabla (u_1 - u_2)(t) \|_{L^2\II} \|  q_2 \|_{L^\infty\II}
 \| \nabla \fy \|_{L^2\II} + \|   {\partial_t^\alpha} (u_1-u_2) (t)  \|_{H^{-1}\II} \| \fy \|_{H^1\II}\big)\\
 \le& c \big(\|  \nabla (u_1 - u_2)(t) \|_{L^2\II}
 + \|   {\partial_t^\alpha} (u_1-u_2) (t)  \|_{H^{-1}\II} \big) \| \nabla \fy(t) \|_{L^2\II}.
\end{split}
\end{equation*}
Since $\|\nabla \varphi(t)\|_{L^2(0,T;L^2(\Omega))}\leq c$, we obtain
\begin{align*}
 \int_0^T((q_1-q_2)\nabla u_1,\nabla\varphi)(t)\d t \le c \big(\|\nabla (u_1 - u_2) \|_{L^2(0,T;L^2\II)}
 + \| {\partial_t^\alpha} (u_1-u_2)\|_{L^2(0,T;H^{-1}\II)} \big).
\end{align*}
This and \eqref{eqn:crucial-identity} give the first estimate.
Next, we decompose the domain $\Omega$ into two disjoint sets $\Omega=\Omega_\rho\cup
\Omega_\rho^c$. with $\Omega_\rho=\{x\in\Omega:{\rm dist}(x,\partial\Omega)\geq\rho\}$
and $\Omega_\rho^c=\Omega\setminus\Omega_\rho$, with $\rho>0$ to be chosen.
On the subdomain $\Omega_\rho$, the $\beta$-positivity condition implies
 \begin{align*}
     &\quad \int_0^T\int_{\Omega_\rho}(q_1-q_2)^2\d x \d t=\rho^{-\beta}\int_0^T\int_{\Omega_\rho }(q_1-q_2)^2\rho^{\beta}\d x\d t\\
     &\leq \rho^{-\beta}\int_0^T\int_{\Omega_\rho }(q_1-q_2)^2\mathrm{dist}(x,\partial\Omega)^{\beta}dx\d t\\
     & \leq c\rho^{-\beta}\int_0^T\int_{\Omega_\rho }(q_1-q_2)^2(q_1(x,t) |\nabla u_1|^2 +(f-\partial_t^\alpha u_1)u_1)\d x\d t\\
     &\leq c\rho^{-\beta}c \big(\|\nabla (u_1 - u_2) \|_{L^2(0,T;L^2\II)}
 + \| {\partial_t^\alpha} (u_1-u_2)\|_{L^2(0,T;H^{-1}\II)} \big).
 \end{align*}
By the box constraint of $\mathcal{A}$, we have
$$\int_0^T\int_{\Omega_\rho^c}(q_1-q_2)^2\d x\d t \leq c_T|\Omega_\rho^c|\leq c\rho.$$
Then the desired result follows by balancing the last two estimates with $\rho$.
\end{proof}

Next we present an alternative conditional stability estimate without the term $\partial_t^\alpha(u_1-u_2)$,
thereby relaxing the temporal regularity assumption on $u(q^\dag)$.
\begin{theorem}\label{thm:stab2}
Let $u_0$, $f$, and $q_i$, $i=1,2$, satisfy the conditions in Assumption \ref{ass:data-1} with $r>d$
and $p>\frac{2r}{\alpha(r-d)}$ and condition \eqref{Cond-q}, and $u_i\equiv u(q_i)$ be the solution to problem \eqref{eqn:fde}.
Then there holds
\begin{align*}
 &\int_0^T\!\! \int_0^t \!\! \int_\Omega \Big(\frac{q_1-q_2}{q_1}(s)\Big)^2  \big(  q_1(s) |\nabla u_1(s)  |^2 +( f(s) -\partial_s^\alpha u_1(s))u_1(s)  \big) \,\d x \,\d s \, \d t\le c\|  \nabla (u_1 - u_2)\|_{L^2(0,T; L^2\II)}.
\end{align*}
Further, if the solution $u_1$ of problem \eqref{eqn:fde} satisfies the $\beta$-positivity condition, then
\begin{align*}
\|q_1-q_2\|_{L^2(0,T;L^2\II)} \le c \|  \nabla (u_1 - u_2) \|_{ {L^2}(0,T; L^2\II)}  ^{\frac{1}{2(1+\beta)}}.
\end{align*}
\end{theorem}
\begin{proof}
By the argument for Theorem \ref{thm:stab1}, it suffices to bound the term $\int _0^T\int_0^t
(\partial_s^\alpha(u_1-u_2)(s), \varphi(s))  \,\d s \,\d t$. By applying integration by
parts in time $s$, since $u_1(0)-u_2(0)=0$, we obtain
\begin{align*}
  \int_0^t (\partial_s^\alpha(u_1-u_2)(s),\varphi(s))  \,\d s
 =&   \int_0^t \big((u_1-u_2)(s), {{_s\partial^\alpha_t}} \fy(s)  \big)\,\d s\\
  &+c_\alpha \int_0^t \big((u_1-u_2)(s), {(t-s)^{-\alpha}\fy(t)}  \big)\,\d s:={\rm I} + {\rm II},
\end{align*}
with $c_\alpha=\frac{1}{\Gamma(1-\alpha)}$ and ${{_s\partial^\alpha_t}} \fy(s) = -c_\alpha\int_s^t
(\xi-s)^{-\alpha}\fy'(\xi)\d\xi$ denoting the right-sided Djrbashian-Caputo fractional derivative. Upon inserting
the test function $\varphi(t) = \frac{q_1-q_2}{q_1} u_1(t)$ into the preceding identity, since $u_1\in
L^\infty((0,T)\times\Omega)$ (cf. Theorem \ref{thm:weak-sol}),
and $\|q_i\|_{L^\infty(0,T;L^\infty(\Omega))}\leq c_1$, by Proposition \ref{prop:reg-2}, we deduce
\begin{align*}
 \| {_s\partial^\alpha_t} \varphi(s) \|_{L^2\II} \le& c \int_s^t (\xi-s)^{-\alpha} (\| u_1'(\xi)  \|_{L^2\II}
 +  \| u_1(\xi)  \|_{L^2\II}) \,\d \xi\le c_\epsilon s^{-\frac\alpha2-\epsilon} + c(t-s),
\end{align*}
for any small $\epsilon>0$.
Thus, choosing $\epsilon\in (0,\frac{1-\alpha}{2})$ leads to
\begin{align*}
|{\rm I}|&\le  \int_0^T  \int_0^t \|\nabla(u_1-u_2)(s)\|_{L^2\II} \| {{_s\partial^\alpha_t}} \fy(s)\|_{H^{-1}\II}\,\d s \,\d t \\
 &\le  c_{\epsilon} \int_0^T  \int_0^t \|\nabla (u_1-u_2)(s)\|_{L^2\II}  ({s^{-\frac\alpha2-\epsilon}}+(t-s))\,\d s \,\d t\\
 &\le c    \|\nabla (u_1-u_2)\|_{L^2(0,T;L^2\II)}.
\end{align*}
Meanwhile, since $q_i\in\mathcal{A}$ and $u\in L^\infty(0,T;L^2(\Omega))$, cf. Theorem \ref{thm:weak-sol},
the bound $ \| \fy \|_{L^\infty(0,T;L^2\II)} \le c$ holds. Hence, by Poincar\'{e}'s inequality,
\begin{equation*}
\begin{split}
|{\rm II}|& \le  \int_0^T  \int_0^t \|(u_1-u_2)(s)\|_{L^2\II} (t-s)^{-\alpha} \,\d s \,\d t
\le c    \|\nabla(u_1-u_2)\|_{L^2(0,T;L^2\II)} .
\end{split}
\end{equation*}
The second assertion follows directly exactly as in Theorem \ref{thm:stab1}, and hence the proof is omitted.
\end{proof}

\begin{remark}\label{rmk:stability}
Under a slightly stronger assumption on problem data, i.e. Assumption \ref{ass:data-2},
we can derive a stability for $u(q_i)\in L^2(0,T;L^2(\Omega))$ using the
Gagliardo-Nirenberg interpolation inequality (e.g., \cite{BrezisMironescu:2018})
\begin{equation*}
  \|u\|_{H^1(\Omega)}\leq c\|u\|_{L^2(\Omega)}^\frac12\|u\|_{H^2(\Omega)}^\frac12.
\end{equation*}
Under Assumption \ref{ass:data-2}, by Proposition \ref{prop:reg-1}, we have the \textit{a priori}
regularity $u(q_i)\in L^2(0,T;H^2(\Omega))$. Then it follows directly from Theorem \ref{thm:stab2} that
\begin{align*}
 &\quad\int_0^T \int_0^t  \int_\Omega \Big(\frac{q_1-q_2}{q_1}(s)\Big)^2  \big(  q_1(s) |\nabla u_1(s)  |^2 +( f(s) -\partial_s^\alpha u_1(s))u_1(s)  \big) \,\d x \d s  \d t
 \le c\| u_1 - u_2\|_{L^2(0,T; L^2\II)}^\frac12.
\end{align*}
Accordingly, if the $\beta$-positivity condition holds, then
\begin{align*}
\|q_1-q_2\|_{L^2(0,T;L^2\II)} \le c \|u_1 - u_2\|_{ {L^2}(0,T; L^2\II)}  ^{\frac{1}{4(1+\beta)}}.
\end{align*}
\end{remark}

The $\beta$-positivity condition plays a central role in deriving the standard $L^2(0,T;L^2(\Omega))$
estimate in Theorems \ref{thm:stab1} and \ref{thm:stab2}. Thus it is important to verify this condition.
Below we give sufficient conditions for the $\beta$-positivity condition, with $\beta=2$ and
$\beta=0$, respectively, for a class of problem data. The main analytic tool is the
maximum principle (see e.g., \cite{LuchkoYamamoto:2017} and \cite[Section 6.5]{Jin:2021book}).
The next two results show the condition for the case of a time-independent diffusion coefficient $q$.
\begin{proposition}\label{prop:beta2-parab}
Let $\Omega$ be a bounded Lipschitz domain, $q\in \Uad$ be time-independent,
$u_0 \in H^2(\Omega)\cap H_0^1(\Omega)$, and $f\in W^{\alpha, p}(0,T;L^2(\Omega))$ with
$p>\frac1\alpha$. Meanwhile, assume that $f \ge c_f >0$ and $\partial_t^\alpha f \le 0$ a.e. in $\Omega
\times[0,T]$, and $u_0\ge0$, $f(0) + \nabla\cdot(q \nabla u_0) \le 0$ a.e. in $\Omega$. Then the
$\beta$-positivity condition holds with $\beta = 2$, with the constant $c$ depending only on $c_0, c_1,
c_f$ and $\Omega$.
\end{proposition}
\begin{proof}
Since $u_0\ge0$ and $f>c_f$, the maximum principle for subdiffusion (see \cite{LuchkoYamamoto:2017}) implies
$u \ge 0$ in ${\Omega}\times[0,T]$. Let $w=\partial_t^\alpha u$. Then it satisfies
\begin{equation}\label{fde-w}
\left\{\begin{aligned}
   \partial_t^\alpha w - \nabla\cdot(q \nabla w) &= \partial_t^\alpha f,&& \mbox{in }\Omega\times(0,T], \\
   w & = 0, &&\mbox{on }\partial\Omega\times(0,T],\\
   w(0) &= f(0) + \nabla\cdot(q \nabla u_0),&&\mbox{in }\Omega.
\end{aligned}\right.
\end{equation}
Since $f\in W^{\alpha, p}(0,T;L^2(\Omega))$, we deduce $\partial_t^\alpha f \in L^p(0,T;L^2\II)$. 
Thus, the system \eqref{fde-w} admits a unique solution $w\in C([0,T];L^2\II)$.
By assumption, $\partial_t^\alpha f\leq0$ in $\Omega\times[0,T]$ and $w(0)\leq0$ in $\Omega$.
Then the maximum principle for subdiffusion  (see \cite{LuchkoYamamoto:2017})  implies $\partial_t^\alpha u= w \le 0$ in  $\Omega\times[0,T]$.
Therefore, there holds
 \begin{align}\label{eqn:est-pos}
   q(x) |\nabla u(x,t)|^2 + (f(x,t)-\partial_t^\alpha u(x,t))u(x,t)
&\ge \min(c_0, c_f)(|\nabla u(x,t)|^2 + u(x,t)).
\end{align}
So it suffices to prove $u(x,t)\ge c\, \text{dist}(x,\partial\Omega)^2$ for $(x,t)\in\Omega\times(0,T]$. For
any fixed $t\in(0,T]$, we have $f(x,t) - \partial_t^\alpha u(x,t)\in L^2(\Omega)$. Now consider the following
boundary value problem
\begin{align}\label{eqn:para-ellip}
 \left\{\begin{aligned}
 -  \nabla \cdot(q\nabla u(t)) &= f(t) - \partial_t^\alpha u(t),\quad\text{in } \Omega,\\
   u(t) &= 0, \quad \text{on }  \partial \Omega.
\end{aligned}\right.
\end{align}
Let $G(x,y)$ be Green's function for the elliptic operator $\nabla
\cdot(q\nabla \cdot)$ with a zero Dirichlet boundary condition. Then $G(x,y)$ is nonnegative (by maximum principle)
and satisfies (\cite[Theorem 1.1]{GruterWidman:1982} and \cite[Lemma 3.7]{Bonito:2017})
$G(x,y) \ge c |x-y|^{2-d}$ for $|x-y|\le \rho(x):=\text{dist}(x,\partial\Omega)$.
Thus, for any $(x,t)\in \Omega\times(0,T]$, there holds
\begin{align*}
u(x,t)  &= \int_\Omega G(x,y)  (f(y,t) - \partial_t^\alpha u(y,t)) \,\d y \ge \int_\Omega G(x,y)  f(y,t)\,\d y \ge c_f \int_\Omega G(x,y)   \,\d y\\
&\ge  c_f \int_{|x-y|<\frac{\rho(x)}{2}} G(x,y)   \,\d y \ge c \int_{|x-y|<\frac{\rho(x)}{2}} |x-y|^{2-d} \, \d y \ge c\rho(x)^2.
 \end{align*}
This completes the proof of the proposition.
\end{proof}\vskip5pt

The next result gives sufficient conditions for the $\beta$-positivity condition with
$\beta=0$, under stronger regularity assumptions on the problem data.

\begin{proposition}\label{prop:beta0-parab}
For some $\mu\in(0,1)$, let $\Omega$ be a bounded $C^{2,\mu}$domain,
$f\in C^{1}([0,T]; C(\overline{\Omega}))\cap C([0,T]; C^{\mu}(\overline{\Omega}))$
with $f \ge c_f >0$, $\partial_t^\alpha f \le 0$ in $\overline{\Omega}\times[0,T]$,
and $u_0 \in C^{2,\mu}(\overline{\Omega})\cap H_0^1(\Omega)$ with $u_0\ge 0$ in $\Omega$.
Moreover, let $q\in \Uad\cap C^{1,\mu}( \overline{\Omega})$ {be time-independent}
with $\| q \|_{C^{1,\mu}(\overline{\Omega})}\le c_2$,
and $f(0) + \nabla\cdot(q \nabla u_0) \le 0$ in $\Omega$.
Then the $\beta$-positivity condition holds with $\beta = 0$, with the constant $c$ only depending
on $c_0, c_1, c_2, c_f$ and $\Omega$.
\end{proposition}
\begin{proof}
By the H\"older regularity estimate (\cite[Theorem 2.1]{Krasnoschok:2016} and \cite[Theorem 7.9]{Jin:2021book}),
we have $u\in C([0,T]; C^{2,\mu}(\overline{\Omega}))$ and $\partial_t^\alpha u \in C([0,T]; C^{\mu}(\overline{\Omega}))$.
The argument of Proposition \ref{prop:beta2-parab} implies $\partial_t^\alpha u \le 0$ for all $(x,t)\in \overline{\Omega}\times [0,T]$, and the
lower bound in \eqref{eqn:est-pos} holds. Next we prove that for any $t\in (0,T]$
\begin{equation}\label{eqn:est-prop2}
|\nabla u(t)|^2+u(t)\geq c >0,\quad \text{ a.e. in } ~\Omega.
\end{equation}
Note that for any $t \in (0,T]$, $u(t)$ solves the boundary value problem \eqref{eqn:para-ellip}
with a $C^{\mu}(\overline{\Omega})$ source $F(t):=f(t)-\partial_t^\alpha u(t)\ge f(t) \ge c_f$,
{and the given assumption ensures that equation \eqref{eqn:para-ellip} holds in a strong sense.}
Then the proof of the assertion \eqref{eqn:est-prop2}  follows from Schauder estimates, Hopf's lemma, and a standard
compactness argument \cite[Lemma 3.3]{Bonito:2017}. Next we sketch the proof of the estimate \eqref{eqn:est-prop2} for completeness.

Assume the contrary of \eqref{eqn:est-prop2}, i.e., for any fixed $t\in(0,T)$, there exists
a sequence $\{q^n\}_{n\geq0}\subset \mathcal{A}$ with $\|q^n\|_{C^{1,\mu}(\overline{\Omega})}
\leq c_2$, such that, for each $n\geq 0$, there exists a point $x_n\in \Omega$ with $|\nabla u(q^n)(x_n,t)|^2
+u(q^n)(x_n,t)\leq n^{-1}$. The classical Schauder estimate \cite[Theorem 6.6]{GilbargTrudinger:1983}
implies $\|u(q^n)(t)\|_{C^{2,\mu}(\overline{\Omega})}\leq c$, for some {constant
$c$ is independent of $n$}. Then by
compactness, up to a subsequence, we have: (i) $q^n$ converges in $C^1(\overline{\Omega})$
to a limit $q^*$; (ii) $u(q^n)(t)$ converges in $C^2(\overline{\Omega})$ to a limit $u^*$ and (iii)
$x_n$ converges to a limit $x^*\in \overline{\Omega}$. Therefore, upon passing to limit, $-q^*\Delta u^* - \nabla
q^*\cdot\nabla u^*=F(t)$ holds on $\Omega$, with $u^*=0$ on $\partial\Omega$, and we have $u^*(x^*)=0$
and $\nabla u^*(x^*)=0$. By the strong maximum principle \cite[Theorem 3.5]{GilbargTrudinger:1983},
$x^*$ lies on the boundary $\partial\Omega$, and the condition $\nabla u^*(x^*)=0$
contradicts Hopf's lemma \cite[Lemma 3.4]{GilbargTrudinger:1983}.
\end{proof}

For a space-time dependent coefficient $q(x,t)$, the argument in Propositions \ref{prop:beta2-parab} and
\ref{prop:beta0-parab} does not work any more: applying the operator $\partial_t^\alpha$ to both
sides of problem \eqref{eqn:fde} does not lead to a tractable identity for $\partial_t^\alpha u$,
due to the nonlocality of $\partial_t^\alpha u$. Nonetheless, if $q^\dag$ is separable, i.e.,
$q(x,t) = a(x) b(t)$, then the $\beta$-positivity condition does hold with $\beta=2$, under suitable
conditions. Below the operator $A:H^2\II\cap H_0^1\II \rightarrow L^2\II$ is defined by $A v =
-\nabla\cdot(a\nabla v)$.
\begin{proposition}\label{lem:beta2-parab-2}
Let $\Omega$ be a bounded Lipschitz domain, $q\in \Uad$, condition \eqref{Cond-q} be fulfilled, and
$q(x,t) = a(x) b(t)$ with smooth $a$ and $b$ such that $b(t)\ge b(0)> 0$ for all $t\in(0,T]$.
Suppose that $u_0 \in D(A^2)$ with $u_0\ge0$ a.e. in $\Omega$,
and $f\in L^p(0,T;D(A))$ with $p>\frac1\alpha$, with $f \ge c_f >0$ and
$\partial_t^\alpha f \le 0$ a.e. in $\Omega\times(0,T)$.
Further, for  $F(t):=f(t)-b(t) A u_0$, there hold $F\le 0$ and $AF\ge 0$ a.e. in $\Omega \times(0,T)$.
Then the $\beta$-positivity condition holds with $\beta = 2$, with the constant $c$ depending only on
$c_0, c_1, c_f$ and $\Omega$.
\end{proposition}
\begin{proof}
Let $w=u-u_0$. Then it satisfies
 \begin{equation}\label{fde-rew}
   \partial_t^\alpha w(t) + b(t) A w(t)= F(t),\quad \forall  t\in(0,T],\quad \text{with}~~w(0) = 0.
\end{equation}
Noting that $F(t) \in L^p(0,T; H^2\II\cap H_0^1\II)$ and applying the operator $A$ to \eqref{fde-rew}, we derive that for $v(t)=Aw(t)$,
 \begin{equation*}
   \partial_t^\alpha v(t) + b(t) A v(t)  =  AF(t), \quad \forall  t\in(0,T],\quad \text{with}~~v(0) = 0.
\end{equation*}
Since $ AF \in L^p(0,T; L^2\II)$ with $p>\frac1\alpha$, there exists a unique weak solution
$v\in L^p(0,T;H^2\II\cap H_0^1\II)$ \cite[Theorem 6.14]{Jin:2021book}.
Moreover, the assumption  $AF \ge 0$ a.e. in $\Omega \times[0,T]$ and
the maximum principle (cf. \cite{LuchkoYamamoto:2017}) imply
$Aw = v\ge 0$  a.e. in $\Omega \times[0,T]$.
This and the assumption $F\le 0$ in $\Omega\times[0,T]$ imply
 \begin{equation*}
   \partial_t^\alpha u(t) = \partial_t^\alpha w(t)   =  F(t)-b(t) A w(t)\le 0\qquad  \text{a.e. in} ~~ \Omega \times(0,T].
\end{equation*}
Next, let the auxiliary function $\uu$ be defined by
 \begin{equation*}
   \partial_t^\alpha \uu(t) + b(0) A \uu(t)  =  f(t), \quad \forall t\in(0,T],\quad \text{with}~~\uu(0) = u_0.
\end{equation*}
Let $\phi= \uu-u$. Then $\phi$ satisfied for all $ t\in(0,T]$
 \begin{equation*}
   \partial_t^\alpha \phi(t) + b(0) A \phi(t)  = (b(t) -b(0))Au(t) = (b(t) -b(0))(Aw(t) + Au_0).
\end{equation*}
with $\phi(0) = 0$. Since $b(t)\le b(0)$ and $Au_0, Aw \ge 0$, we apply the maximum principle
 (see \cite{LuchkoYamamoto:2017})  again to derive $\phi\le 0$ in $\Omega\times(0,T)$, i.e. $\uu\le u$
in $\Omega\times (0,T)$. Therefore, there holds
 \begin{align*}
   q(x) |\nabla u(x,t)|^2 + (f(x,t)-\partial_t^\alpha u(x,t))u(x,t)
&\ge  (c_0,c_f) \min(|\nabla u(x,t)|^2,  \uu(x,t)).
\end{align*}
Finally, repeating the argument for Proposition \ref{prop:beta2-parab} on the function
$\uu$ leads to the $\beta$-positivity condition  with $\beta = 2$.
\end{proof}

\section{Regularized problem and the numerical approximation}\label{sec:reg}
In this section, we propose the continuous formulation of the reconstruction approach
based on the regularized output least-squares method and develop a fully discrete scheme for
practical implementation. The error analysis of the discrete approximations is given
in Section \ref{sec:err}.

\subsection{Output least-square formulation}
To recover the diffusion coefficient $q(x,t)$, we employ an output
least-squares formulation with an $H^1(\Omega\times(0,T))$ seminorm penalty (with the notation $\nabla_{x,t}$
denoting the space and time gradient):
\begin{equation}\label{eqn:ob}
    \min_{q \in \Uad} J_\gamma(q;z^\delta)=\tfrac12   \|u( q) - z^\delta\|_{L^2(0,T;L^2(\Omega))}^2
    + \tfrac{\gamma}{2}\|\nabla_{x,t} q \|_{L^2(0,T; L^2(\Omega)) }^2,
\end{equation}
with $u(q)$ satisfying $u(q)(0)=u_0$
\begin{equation}\label{eqn:var}
   (\Dal u(q)(t),\phi)+(q(t)\nabla u(q)(t),\nabla \phi) = (f,\phi),\quad \forall  \phi\in H_0^1(\Omega),\ t\in(0,T].
\end{equation}
The admissible set $\Uad$ for $q(x,t)$ is given in \eqref{eqn:Admset}. The scalar $\gamma>0$ is the
regularization parameter, controlling the strength of the penalty \cite{EnglHankeNeubauer:1996,ItoJin:2015}.
The $H^1(\Omega\times(0,T))$ seminorm penalty is suitable for recovering a spatially-temporally smooth
diffusion coefficient, and it is essential for the error analysis in Section \ref{sec:err}. With this penalty
term, the numerically recovered  diffusion coefficient admits a uniformly bounded (space and time) gradient
in the $L^2(0,T;L^2(\Omega))$ norm, dependent of the regularization parameter $\gamma$ (cf. Lemma \ref{lem:err-2}),
which is needed in the proof of Theorem \ref{thm:error-q}. The dependence of the functional $J_\gamma$
on $z^\delta$ will be suppressed whenever there is no confusion. To ensure the well-posedness of problem
\eqref{eqn:ob}--\eqref{eqn:var}, we make the following assumption on the given problem data.

\begin{assumption}\label{ass:data1}
$u_0\in L^2(\Omega) $, and $f\in L^2(0,T;H^{-1}(\Omega))$.
\end{assumption}

Note that Assumption \ref{ass:data1} and the condition $q\in\mathcal{A}$ in the regularized formulation
are weaker than that in Theorem \ref{thm:weak-sol}. Nonetheless, problem \eqref{eqn:var} does
has a unique weak solution $u$, which can be proved using the standard Galerkin procedure, where $_0I_t^\beta$
denotes the Riemann-Liouville fractional integral of order $\beta$. For a detailed
proof, see, e.g., \cite[Chapter 4]{KubicaYamamoto:2020} and \cite[Section 6.1]{Jin:2021book}.

\begin{lemma}\label{lem:exist-weak}
For any $q\in \Uad$, under Assumption \ref{ass:data1}, problem \eqref{eqn:var} has a unique weak solution
$u(q) \in L^2(0,T;H_0^1\II)$ with $_0I_t^{1-\alpha}(u - u_0) \in {_0H^1}(0,T; H^{-1}\II)$ and
\begin{equation*}
 \|  u(q) \|_{L^2(0,T;H_0^{1}\II)} \le c (\| u_0 \|_{L^2(\Omega)} + \| f  \|_{L^2(0,T; H^{-1}\II)}).
\end{equation*}
\end{lemma}

The following continuity result for the forward map $u(q)$ is useful.
\begin{lemma}\label{lem:conti-p2s}
Let Assumption \ref{ass:data1} be fulfilled, and the sequence $\{q^n\}\subset \Uad$ converge to $q\in \Uad$ in $L^1(\Omega\times(0,T))$ and a.e., and let $u(q^n)$ and $u(q)$ solve
problem \eqref{eqn:var} with the diffusion coefficients $q^n$ and $q$, respectively. Then
\begin{equation*}
  \lim_{n\to\infty} \|u(q)-u(q^n)\|_{L^2(0,T;H^1(\Omega))} = 0.
\end{equation*}
\end{lemma}
\begin{proof}
Let $v^n = u(q)-u(q^n)$. Then it satisfies $v^n(0)=0$ and
\begin{equation*}
  \partial_t^\alpha v^n - \nabla\cdot (q^n\nabla v^n) = \nabla\cdot((q-q^n)\nabla u(q)),\quad \forall t\in(0,T].
\end{equation*}
Then by Lemma \ref{lem:exist-weak} and the definition of the $H^{-1}(\Omega)$-norm, we obtain
\begin{align*}
  \|v^n\|_{L^2(0,T;H^1(\Omega))} & \leq c\|\nabla\cdot((q-q^n)\nabla u(q))\|_{L^2(0,T;H^{-1}(\Omega))}
    \leq c\|(q-q^n)\nabla u(q)\|_{L^2(0,T;L^2(\Omega))}.
\end{align*}
Let $\phi^n = |q-q^n|^2  |\nabla u(q)|^2  $, then $\phi^n\rightarrow0$ almost everywhere (a.e.),
since $q^n\to q$ a.e., and further, since $q, q^n \in \mathcal{A}$, we have
$0\leq \phi^n \le 4c_1^2   |\nabla u(q)|^2  \in L^1(0,T;L^1(\Omega)).$ Then, Lebesgue's dominated
convergence theorem \cite[Theorem 1.9]{EvansGariepy:2015} implies
\begin{align*}
&\lim_{n\rightarrow\infty} \|(q-q^n)\nabla u(q)\|_{L^2(0,T;L^2(\Omega))}^2 = \lim_{n\rightarrow\infty} \int_0^T\int_\Omega \phi^n(x,t)\,\d x\d t
 = \int_0^T\int_\Omega  \lim_{n\rightarrow\infty} \phi^n(x,t)\,\d x \d  t = 0,
\end{align*}
which shows the desired estimate.
\end{proof}

Lemma \ref{lem:conti-p2s} implies that the forward map
$q\in H^1((0,T)\times \Omega) \rightarrow u(q) \in L^2(0,T;H^1(\Omega))$
is weakly sequential closed. Then a standard argument \cite[Theorem 1]{SeidmanVogel:1989} leads to the
existence of a minimizer to problem \eqref{eqn:ob}--\eqref{eqn:var}, given in the next theorem.

\begin{theorem}\label{thm:ex}
Under Assumption \ref{ass:data1}, there exists at least one minimizer to problem \eqref{eqn:ob}--\eqref{eqn:var}.
\end{theorem}

Using Lemma \ref{lem:conti-p2s}, the following continuity results follow from a standard
compactness argument \cite{EnglHankeNeubauer:1996,ItoJin:2015}.
\begin{theorem}
Under Assumption \ref{ass:data1}, the following two statements hold.
\begin{itemize}
  \item[$\rm(i)$] Let the sequence $\{z_j\}_{j\geq 1}$ be convergent to $z^*$ in $L^2(0,T;L^2(\Omega))$, and $q_j^*\in\Uad$
  the corresponding minimizer to $J_\gamma(\cdot;z_j)$. Then $\{q_j^*\}_{j\geq1}$
  contains a subsequence convergent to a minimizer of $J_\gamma(\cdot;z^*)$ over $\mathcal{A}$ in $H^1(\Omega\times(0,T))$.
  \item[$\rm(ii)$] Let $\{\delta_j\}_{j\geq1}\subset \mathbb{R}_+$ with $\delta_j\to0$, $\{z^{\delta_j}\}_{j\geq1}
  \subset L^2(0,T;L^2(\Omega))$ be a sequence satisfying $\|z^{\delta_j}-z^*\|_{L^2(0,T;L^2(\Omega))}
  =\delta_j$ for some exact data $z^*$, and $q_j^*$ be a minimizer to $J_{\gamma_j}
  (\cdot;z^{\delta_j})$ over $\mathcal{A}$. If the sequence $\{\gamma_j\}_{j\geq1}\subset\mathbb{R}_+$ satisfies
  $\lim_{j\to\infty}\gamma_j = 0$ and $\lim_{j\to \infty}\frac{\delta_j^2}{\gamma_j} =0$,
 then the sequence $\{q_j^*\}_{j\geq 1}$ contains a convergent subsequence and
      the limit of every convergent subsequence is a minimum-$H^1(\Omega\times(0,T))$ seminorm solution.
\end{itemize}
\end{theorem}

\begin{remark}
Under the $\beta$-positivity condition, the inverse problem has a unique solution, so the minimum-seminorm solution is unique.
Then the standard subsequence argument shows that in (ii), actually the whole sequence converges.
\end{remark}

\subsection{Numerical approximation}
Now we describe the discretization of problem \eqref{eqn:ob}--\eqref{eqn:var}, based on the
Galerkin finite element method (FEM) in space (cf. \cite{Thomee:2006}) and backward Euler
convolution quadrature (CQ) in time due to \cite{Lubich:1986}. First we recall the Galerkin
FEM approximation. Let $\mathcal{T}_h$ be a shape regular quasi-uniform triangulation of the
domain $\Omega$ into $d$-simplexes, denoted by $K$, with a mesh size $h$. Over $\mathcal{T}_h$,
we define continuous piecewise linear finite element spaces $X_h$ and $V_h$, respectively, by
\begin{align*}
  X_h&= \left\{v_h\in H_0^1(\Omega):\ v_h|_K\mbox{ is a linear function}\ \forall\, K \in \mathcal{T}_h\right\},\\
  V_h&= \left\{v_h\in H^1(\Omega):\ v_h|_K  \mbox{ is a linear function}\ \forall\,  K \in \mathcal{T}_h\right\}.
\end{align*}
The spaces $X_h$ and $V_h$ will be employed to approximate the state $u$ and the diffusion
coefficient $q$, respectively. Now we introduce useful operators on the spaces $X_h$ and
$V_h$. We define the $L^2(\Omega)$ projection $P_h:L^2(\Omega)\to X_h$ by
\begin{equation*}
     (P_h v,\chi) =(v,\chi) , \quad \forall v\in L^2(\Omega),\chi\in X_h.
\end{equation*}
It satisfies the following error estimate \cite[p. 32]{Thomee:2006}: for any $s\in[1,2]$
\begin{equation}\label{eqn:proj-L2-error}
  \|P_hv-v\|_{L^2(\Omega)} + h\|\nabla(P_hv-v)\|_{L^2(\Omega)}\leq h^s\|v\|_{H^s(\Omega)},\quad \forall v\in H^s(\Omega)\cap H_0^1(\Omega).
\end{equation}
Let $\mathcal{I}_h$ be the Lagrange interpolation operator associated with the finite element space $V_h$. It satisfies
the following error estimates for $s=1,2$ and $1 \le p\le \infty$ (with $sp>d$) \cite[Theorem 1.103]{ern-guermond}:
\begin{align*}
  \|v-\mathcal{I}_hv\|_{L^p(\Omega)} + h\|v-\mathcal{I}_hv\|_{W^{1,p}(\Omega)} & \leq ch^s \|v\|_{W^{s,p}(\Omega)}, \quad \forall v\in W^{s,p}(\Omega).\label{eqn:int-err-inf}
\end{align*}
Further, for any $q\in \mathcal{A}$, we define a discrete operator $A_h(q(t)):X_h\to X_h$ by
\begin{equation}\label{eqn:Ah}
  (A_h(q(t))v_h,\chi)=(q(t)\nabla v_h,\nabla \chi),\quad \forall v_h,\chi\in X_h.
\end{equation}

Next we describe time discretization. We partition the interval $[0,T]$ uniformly,
with grid points $t_n=n\tau$, $n=0,\ldots,N$, and a time step size $\tau=T/N$. The fully
discrete scheme for problem \eqref{eqn:fde} reads: Given $U_h^0=P_hu_0\in X_h$, find $U_h^n
\in X_h$ such that
\begin{align}\label{eqn:fully-0}
  (\bar \partial_\tau^\alpha (U_h^n-U_h^0),\chi)+(q(t_n) \nabla U_h^n, \nabla \chi)= (f^n,\chi),\quad \forall\chi\in X_h, \,\, n=1,2,\ldots,N,
\end{align}
where {$f^n=f(t_n)$} 
and $\bar\partial_\tau^\alpha \varphi^n$ denotes the
backward Euler CQ approximation (with $\varphi^j=\varphi(t_j)$):
\begin{equation}\label{eqn:CQ-BE}
  \bar\partial_\tau^\alpha \varphi^n = \tau^{-\alpha} \sum_{j=0}^nb_j^{(\alpha)} ,\quad\mbox{ with } (1-\xi)^\alpha=\sum_{j=0}^\infty b_j^{(\alpha)}\xi^j.
\end{equation}
Note that the weights $b_j^{(\alpha)}$ are given explicitly by
$b_j^{(\alpha)} = (-1)^j\frac{\Gamma(\alpha+1)}{\Gamma(\alpha-j+1)\Gamma(j+1)}$, and thus
\begin{equation*}
 b_j^{(\alpha)} = (-1)^j(j!)^{-1}\alpha(\alpha-1)\cdots(\alpha-j+1),\quad j= 1,2,\ldots,
\end{equation*}
from which it can be verified directly that $b_0^{(\alpha)}=1$ and $b_j^{(\alpha)}<0$ for $j\geq 1$.
Using the operator $A_h(q(t_n))$, the fully discrete scheme \eqref{eqn:fully-0}
can be rewritten as
\begin{equation*}
  \bar \partial_\tau^\alpha  (U_h^n-U_h^0) -A_h(q(t_n)) U_h^n = P_hf^n, \quad n=1,2,\ldots,N.
\end{equation*}

{We use extensively the norm $\|\cdot\|_{\ell^p(X)}$, $1\leq p<\infty$,
for a finite sequence $(u^n)_{n=1}^N\subset X$ (for a Banach space $X$ equipped with the norm $\|\cdot\|_X$):
\begin{equation*}
  \|(u^n)_{n=1}^N\|_{\ell^p(X)} = \Big(\tau \|u^n\|_X^p\Big)^\frac1p.
\end{equation*}}
Now we are ready to give the fully discrete scheme for problem \eqref{eqn:ob}--\eqref{eqn:var}.
Let $z_n^\delta=\tau^{-1}\int_{t_{n-1}}^{t_n}z^\delta(t) \d t$.
Then the fully discrete formulation for problem \eqref{eqn:ob}--\eqref{eqn:var} is given by
\begin{align}\label{eqn:ob-disc}
    \min_{q_{h,\tau}\in \Uad_{h,\tau}} J_{\gamma,h,\tau}(q_{h,\tau})&=\tfrac12 \|(U_h^n(q_{h,\tau}) - z_n^\delta)_{n=1}^N\|_{\ell^2(L^2(\Omega))}^2
    \\&+ \tfrac{\gamma}2\big(\|(\nabla q_h^n)_{n=1}^N \|_{\ell^2(L^2(\Omega)) }^2 +\|(\partial_\tau q_h^n)_{n=1}^{N}  \|_{\ell^2(L^2(\Omega)) }\big),\nonumber
\end{align}
subject to $U_h^n\equiv U_h^n(q_{h,\tau})\in X_h$ satisfying $U_h^0=P_hu_0$ and
\begin{align}\label{eqn:fully}
  \bar \partial_\tau^\alpha (U_h^n-U_h^0)+ A_h(q_h^n) U_h^n = P_hf^n, \quad n=1,2,\ldots,N.
\end{align}
The discrete admissible set $\Uad_{h,\tau}$ is taken to be
\begin{equation*}
    \Uad_{h,\tau}=\{ q_{h,\tau}=(q_h^n)_{n=1}^N:q_h^n\in V_h,~c_0\le q_h^n \le c_1,\,\, 1\le n\le N \}.
\end{equation*}
Note that we approximate the conductivity $q$ by a finite element function in space and piecewise
constant function in time, and in the discrete objective function $J_{\gamma,h,\tau}$, we approximate the
first-order time-derivative in the penalty by backward difference. Problem
\eqref{eqn:ob-disc}--\eqref{eqn:fully} is a finite-dimensional nonlinear optimization problem with
PDE and box constraints, and can be solved efficiently, e.g., (projected) conjugate gradient method.
The existence of a discrete minimizer $q_{h,\tau}^*=(q_h^{n*})_{n=1}^N\in \Uad_{h,\tau}$ is direct,
in view of the norm equivalence in finite-dimensional spaces.
\begin{theorem}\label{thm:ex-fully}
Under Assumption \ref{ass:data1}, there exists at least one minimizer $q_{h,\tau}^*\in \Uad_{h,\tau}$ to problem \eqref{eqn:ob-disc}--\eqref{eqn:fully}.
\end{theorem}

\section{Error analysis}\label{sec:err}
In this section, we derive an error bound for the approximations $q_{h,\tau}^*\in\Uad_{h,\tau}$ in terms of the
noise level $\delta$, the regularization parameter $\gamma$, and the discretization parameters $h$ and
$\tau$. The delicate interplay between different parameters and limited regularity of the solution and
problem data represent the main challenges in the analysis. The error estimate in Theorem \ref{thm:error-q} involves
the weight involving $q^\dag(t_n)|\nabla u(t_n)  |^2 +( f(t_n) -\partial_t^\alpha u(t_n))u(t_n)$, which
arises naturally in the stability analysis. The proof relies crucially on the choice of the test function
$\varphi^n=\frac{q^\dag(t_n)-q_h^{n*}}{q^\dag}u$, which is inspired by the conditional stability analysis in
Section \ref{sec:stab}, cf. the proofs of Theorems \ref{thm:stab1} and \ref{thm:stab2}.

\begin{assumption}\label{ass:zdelta}
$q^\dag\in C([0,T]; H^2\II)$.
\end{assumption}

\begin{theorem}\label{thm:error-q}
Let $q^\dag$ be the exact diffusion coefficient, $u\equiv u(q^\dag)$ the solution to problem \eqref{eqn:var},
and $q_{h,\tau}^*\in \Uad_{h,\tau}$ the solution to problem \eqref{eqn:ob-disc}--\eqref{eqn:fully}. Then
under Assumptions \ref{ass:data-2} and \ref{ass:zdelta}, there holds
\begin{align*}
 &\tau^2 \sum_{m=1}^N \sum_{n=1}^m \int_\Omega \Big(\frac{q^\dag(t_n)-q_h^{n*}}{q^\dag(t_n)}\Big)^2
  \big(q^\dag(t_n)|\nabla u(t_n)  |^2 +( f(t_n) -\partial_t^\alpha u(t_n))u(t_n)\big)\,\d x\\
 \le& c(h \gamma^{-1}\eta+ h\gamma^{-\frac12} +  h^{-1}\gamma^{-\frac12}\eta + \gamma^{-\frac12}\eta )\eta,
\end{align*}
with $\ell_N= \ln (N+1)$ and
\begin{align*}
\eta=  \begin{cases}
  c  (\tau^{\min(1,\frac12+\alpha)}  + h^2 + \delta+\gamma^{\frac12}),&\alpha\neq1/2;\\
   c (\tau\ell_N^{\frac12}  + h^2 + \delta^2+\gamma^{\frac12}),&\alpha=1/2.
\end{cases}
\end{align*}
\end{theorem}

The proof of Theorem \ref{thm:error-q} is technical and lengthy, and requires several technical estimates, especially
nonstandard nonsmooth data estimates for the discrete scheme for problem \eqref{eqn:fde}. Due to the time-dependence of
the elliptic operator $A(t)$, the requisite estimates are still unavailable, and we develop them in Section
\ref{ssec:techest} in the appendix.

\subsection{Basic estimates}\label{ssec:basic}

The analysis requires two basic estimates (which in turn depend on nonsmooth data estimates in Section
\ref{ssec:techest}). The first result gives an \textit{a priori} bound on $\nabla q_{h,\tau}^*$ and
$\partial_\tau q_{h,\tau}^*$ of the discrete minimizer $q_{h,\tau}^*$ and an error bound on the state approximation
$U_h^n(q_{h,\tau}^*)$. This result will play a crucial role in the proof of Theorem \ref{thm:error-q}.
\begin{lemma}\label{lem:err-2}
Let $q^\dag$ be the exact coefficient and $u\equiv u(q^\dag)$ the solution to problem \eqref{eqn:var}.
Let $q_{h,\tau}^*\in \Uad_{h,\tau}$ be the solution to problem \eqref{eqn:ob-disc}--\eqref{eqn:fully}, and
$\{U_h^n(q_{h,\tau}^*)\}_{n=1}^N$ the fully discrete solution to problem \eqref{eqn:fully}. Then under
Assumptions \ref{ass:data-2} and \ref{ass:zdelta}, with $\ell_N=\ln(N+1)$, there holds
\begin{align*}
   \| (U_h^n(q_{h,\tau}^*) - u(t_n))_{n=1}^N \|_{\ell^2(L^2(\Omega))}^2 + \gamma \|& (\nabla q_{h,\tau}^{n*})_{n=1}^N\|_{\ell^2(L^2(\Omega))}^2
  +  \gamma \| (\partial_\tau  q_{h,\tau}^{n*})_{n=1}^N \|_{\ell^2(L^2(\Omega))}^2\nonumber\\
 &  \le  \begin{cases}
  c  (\tau^{\min(2,1+2\alpha)}  + h^4 + \delta^2+\gamma),&\alpha\neq1/2;\\
   c (\tau^2\ell_N  + h^4 + \delta^2+\gamma),&\alpha=1/2.
\end{cases}
\end{align*}
\end{lemma}
\begin{proof}
First we bound $\|(u(t_n) - z_n^\delta)_{n=1}^N\|_{\ell^2(L^2(\Omega))}^2$. Under the given assumption,
we have the \textit{a priori} regularity $u\in C([0,T];L^2(\Omega))$, and thus $u(t_n)$ is well defined.
Let $u_n = \tau^{-1}\int_{t_{n-1}}^{t_n} u(t)\, \d t$. Then $u(t_n)-u_n = \tau^{-1}\int_{t_{n-1}}^{t_n}
\int_t^{t_n}u'(s)\d s\d t$. For $n\ge 2$, the regularity estimate \eqref{reg-fde-2} implies
\begin{align*}
  \|u(t_n)-u_n\|_{L^2(\Omega)} &\leq \tau^{-1}\int_{t_{n-1}}^{t_n}\int_t^{t_n} \|u'(s)\|_{L^2(\Omega)}\d s\d t
   \leq \tau^{-1}\int_{t_{n-1}}^{t_n}\int_t^{t_n}s^{\alpha-1}\d s\d t \le c \tau t_{n-1}^{\alpha-1}.
\end{align*}
Similarly, we have  $\|u(\tau)-u_1\|_{L^2(\Omega)}\le c\tau^\alpha$. Consequently, we deduce
\begin{align*}
  \|(u(t_n) - u_n)_{n=1}^N \|_{\ell^2(L^2\II)}^2 \le c\Big( \tau^{1+2\alpha} + \tau  \sum_{n=2}^N \tau^2 t_n^{2\alpha-2} \Big)
 \le \begin{cases}
 c \tau^{\min(2,1+2\alpha)} ,&\alpha\neq1/2;\\
 c \tau^2\ell_N ,&\alpha=1/2.
\end{cases}
\end{align*}
Meanwhile, by the Cauchy-Schwarz inequality, $\tau|u_n|^2\leq \int_{t_{n-1}}^{t_n}u(t)^2\d t$. This, the definition
of $\delta$ and the stability estimate $\|(u_n - z_n^\delta)_{n=1}^N\|_{\ell^2(L^2(\Omega))} \le
\|u(t) - z^\delta(t)\|_{L^2(0,T;L^2(\Omega))}\leq \delta$ imply
\begin{equation}\label{eqn:err-time}
  \|(u(t_n) - z_n^\delta)_{n=1}^N\|_{\ell^2(L^2(\Omega))}^2 \leq \begin{cases}
 c \tau^{\min(2,1+2\alpha)} +\delta^2,&\alpha\neq1/2;\\
 c \tau^2\ell_N + \delta^2 ,&\alpha=1/2.
\end{cases}
\end{equation}
Next by the minimizing property of $q_{h,\tau}^*\in \Uad_{h,\tau}$ and $\hat q_{h,\tau} =
(\mathcal{I}_h q^\dag(t_n))_{n=1}^N \in \Uad_{h,\tau}$, we deduce
$$J_{\gamma,h,\tau}(q_{h,\tau}^*)\leq J_{\gamma,h,\tau}(\hat q_{h,\tau}).$$
By the triangle inequality, we derive
\begin{equation*}
\begin{aligned}
\|(U_h^n(q_{h,\tau}^*) -  u(t_n))_{n=1}^N\|_{\ell^2(L^2(\Omega))}^2
 \le  c \|(U_h^n(q_{h,\tau}^*) - z^\delta_n)_{n=1}^N\|_{\ell^2(L^2(\Omega))}^2
  + c  \| (z^\delta_n-u(t_n))_{n=1}^N \|_{\ell^2(L^2(\Omega))}^2 .
\end{aligned}
\end{equation*}
The preceding two inequalities imply
\begin{align*}
   & \| (U_h^n(q_{h,\tau}^*) - u(t_n))_{n=1}^N \|_{\ell^2(L^2(\Omega))}^2 + \gamma \|(\nabla q_h^{n*})_{n=1}^N\|_{\ell^2(L^2(\Omega))}^2
  +  \gamma \|(\bar \partial_\tau q_{h}^{n*})_{n=1}^N\|_{\ell^2(L^2(\Omega))}^2 \\
  \leq & c\|(U_h^n(\mathcal{I}_h q^\dag) - z^\delta_n)_{n=1}^N\|_{\ell^2(L^2(\Omega))}^2 +
  c \gamma \|(\nabla \mathcal{I}_h q^\dag(t_n))_{n=1}^N\|_{\ell^2(L^2(\Omega))}^2\\
    &+  c\gamma \|(\bar \partial_\tau  \mathcal{I}_h q^\dag(t_n))_{n=1}^N\|_{\ell^2(L^2(\Omega))}^2
   + c  \| (z^\delta_n-u(t_n))_{n=1}^N \|_{\ell^2(L^2(\Omega))}^2.
\end{align*}
Since $|\nabla q^\dag| + |\partial_t q^\dag |\le c$ by Assumption
\ref{ass:data-2}, the property of the interpolation operator $\mathcal{I}_h$ implies
$$\|(\nabla \mathcal{I}_hq^\dag(t_n))_{n=1}^N\|_{\ell^2(L^2(\Omega))}^2 + \|(\bar \partial_\tau \mathcal{I}_h q^\dag(t_n))_{n=1}^N\|_{\ell^2(L^2(\Omega))}^2 \leq c.$$
Meanwhile, by the triangle inequality and Lemma \ref{lem:err-1}, we deduce
\begin{align*}
 \|U_h^n(\mathcal{I}_hq^\dag) - z^\delta_n\|_{L^2(\Omega)}^2 &\le 2\|U_h^n(\mathcal{I}_hq^\dag) - u(t_n)\|_{L^2(\Omega)}^2+2\|u(t_n) - z^\delta_n\|_{L^2(\Omega)}^2\\
& \le c(\tau t_n^{\alpha-1} + h^2)^2 + c\|u(t_n) - z^\delta_n\|_{L^2(\Omega)}^2,
\end{align*}
Consequently, combining the preceding estimates with \eqref{eqn:err-time} we derive
\begin{align*}
  \|  ( U_h^n(\mathcal{I}_hq^\dag) - z^\delta_n ) _{n=1}^N \|_{\ell^2(L^2(\Omega))}^2
   &\le { \begin{cases}
  c  (\tau^{\min(2,1+2\alpha)}  + h^4 + \delta^2+\gamma),&\alpha\neq1/2;\\
   c (\tau^2\ell_N  + h^4 + \delta^2+\gamma),&\alpha=1/2.
\end{cases}}
\end{align*}
This completes the proof of the lemma.
\end{proof}

Next we give a bound on the backward Euler CQ approximation of the discrete test function $\varphi$.
\begin{lemma}\label{lem:CQ}
Let $q^\dag$ be the exact coefficient, and $u\equiv u(q^\dag)$ the solution to problem \eqref{eqn:fde}.
Then for $\fy^m=\frac{q^\dag(t_m)-q_h^{m*}}{q^\dag (t_m)}u(t_m)$, there hold for $1\le j\le N$
\begin{align*}
 \tau \sum_{m=j}^N \| \tau^{-\alpha} \sum_{n=j}^m b_{n-j}^{(\alpha)} P_h(\fy^n - \fy^m) \|_{L^2(\Omega)}^2
\le& \begin{cases}
  c  \gamma^{-1} (\tau^{\min(2,1+2\alpha)}  + h^4 + \delta^2+\gamma),&\alpha\neq1/2;\\
   c \gamma^{-1} (\tau^2\ell_N  + h^4 + \delta^2+\gamma),&\alpha=1/2.
\end{cases}
\end{align*}
\end{lemma}
\begin{proof}
By the associativity of backward Euler CQ, i.e., $\bar\partial_\tau^\alpha\fy^n
=\bar\partial_\tau^{\alpha-1}\bar\partial_\tau \fy^n$, if $\fy^0=0$, then there holds
\begin{align*}
   {\rm I}_m:=\tau^{-\alpha} \sum_{n=j}^m b_{n-j}^{(\alpha)} P_h(\fy^n -\fy^m) & = \tau^{1-\alpha} \sum_{n=j}^{ {m-1}} b_{n-j}^{(\alpha-1)}\tfrac{P_h\fy^n-P_h\fy^{n+1}}{\tau}.
\end{align*}
Thus, the $L^2(\Omega)$-stability of $P_h$ and the definition of $\fy^n$ imply  
\begin{align*}
   &\quad \tau^{-1}\|P_h(\fy^n-\fy^{n+1})\|_{L^2(\Omega)} \leq \tau^{-1}\|\fy^n - \fy^{n+1}\|_{L^2(\Omega)}\\
    \leq & \| u(t_{n+1})\bar \partial_\tau \tfrac{q^\dag(t_{n+1})-q_h^{n+1*}}{q^\dag (t_{n+1})}\|_{L^2(\Omega)}
    + \| \tfrac{q^\dag(t_{n})-q_h^{n*}}{q^\dag (t_{n})}  \bar \partial_\tau  u(t_{n+1})\|_{L^2(\Omega)}\\
 \leq &\| u(t_{n+1}) \|_{L^\infty\II}
 \|\bar \partial_\tau \tfrac{q^\dag(t_{n+1})-q_h^{n+1*}}{q^\dag (t_{n+1})}\|_{L^2(\Omega)}
    + \| \tfrac{q^\dag(t_{n})-q_h^{n*}}{q^\dag (t_{n})} \|_{L^\infty\II} \|\bar \partial_\tau  u(t_{n+1})\|_{L^2(\Omega)}.
\end{align*}
Since Assumption \ref{ass:data-2} holds, we have $\| u(t_{n+1}) \|_{L^\infty\II}\leq c $ and it follows from
$q_{h,\tau}^* ,q\in\mathcal{A} $ that
\begin{align*}
 \|{\rm I}_m\|_{L^2(\Omega)}
 &\leq c\tau^{1-\alpha} \sum_{n=j}^{m-1} |b_{n-j}^{(\alpha-1)}| \big(
 \|\bar \partial_\tau  q_h^{n+1*} \|_{L^2(\Omega)} +
  \|\bar \partial_\tau  q^\dag(t_{n+1}) \|_{L^2(\Omega)}  +  \|\bar \partial_\tau  u(t_{n+1})\|_{L^2(\Omega)} \big)\\
    &\le c\tau  \sum_{n=j}^{m-1} t_{n-j+1}^{-\alpha}
 \|\bar \partial_\tau  q_h^{n+1*} \|_{L^2(\Omega)} + c\tau  \sum_{n=j}^{m-1} t_{n-j+1}^{-\alpha}
    +  \tau  \sum_{n=j}^{m-1} t_{n-j+1}^{-\alpha} \|\bar \partial_\tau  u(t_{n+1})\|_{L^2(\Omega)} .
\end{align*}
where the last step follows from $|b_j^{(\alpha-1)}|\leq c(j+1)^{-\alpha}$ \cite[Exercise 6.16]{Jin:2021book}.
Note that $ c\tau  \sum_{n=j}^{m-1} t_{n-j+1}^{-\alpha}\leq ct_{n-m}^{1-\alpha}\leq c$. Then Young's inequality implies
\begin{align*}
 \tau \sum_{m=j}^N \bigg(\tau  \sum_{n=j}^{m-1} t_{n-j+1}^{-\alpha}
 \|\bar \partial_\tau  q_h^{*n+1} \|_{L^2(\Omega)} \bigg)^2
 &\leq c \bigg(\tau \sum_{n=j}^N   t_{n-j+1}^{-\alpha}   \bigg)^2  \bigg( \tau \sum_{n=j}^N  \|\bar \partial_\tau  q_h^{*n} \|_{L^2(\Omega)}^2  \bigg)\\
 &\le   \begin{cases}
  c  \gamma^{-1} (\tau^{\min(2,1+2\alpha)}  + h^4 + \delta^2+\gamma),&\alpha\neq1/2;\\
   c \gamma^{-1} (\tau^2\ell_N  + h^4 + \delta^2+\gamma),&\alpha=1/2.
\end{cases}
\end{align*}
Meanwhile, the regularity estimate $\|\partial_t u(t)\|_{L^2\II}\le c t^{\alpha-1}$ from \eqref{reg-fde-2}
 and the argument of \cite[Lemma 4.6]{JinZhou:2021sicon} imply
$\tau  \sum_{n=j}^{m-1} t_{n-j+1}^{-\alpha} \|\bar \partial_\tau  u(t_{n+1})\|_{L^2(\Omega)} \le
c_\epsilon t_j^{-\epsilon},$ for any small $\epsilon \in (0,\min(\frac12,1-\alpha))$. Consequently,
\begin{align*}
 \tau \sum_{m=j}^N \bigg(\tau  \sum_{n=j}^{m-1} t_{n-j+1}^{-\alpha}
 \|\bar \partial_\tau  u(t_{n+1})\|_{L^2(\Omega)} \bigg)^2\
 &\leq c \tau \sum_{m=j}^N   t_j^{-2\epsilon} \le c.
\end{align*}
This completes the proof of the lemma.
\end{proof}

\subsection{The convergence rate}
With the basic estimates in Lemmas \ref{lem:err-2} and \ref{lem:CQ}, we can prove
Theorem \ref{thm:error-q}. The proof relies on a novel choice of the test function $\varphi^n$,
directly inspired by the conditional stability analysis in Section \ref{sec:stab}, and
maximal regularity estimates. Hence, it is still lengthy, and is divided into several steps.\medskip

\noindent\textbf{Proof of Theorem \ref{thm:error-q}.}
The proof employs the following identity, analogous to \eqref{eqn:crucial-identity},
\begin{align*}
((q^\dag(t_n)-q_h^{n*})\nabla u(t_n),\nabla\fy^n)
= &\frac{1}{2}\int_\Omega \Big(\frac{q^\dag(t_n)-q_h^{n*}}{q^\dag(t_n)}\Big)^2\big( q^\dag(t_n)|\nabla u(t_n)  |^2 +(f(t_n)-\partial_t^\alpha u(t_n))u(t_n)\big)\,\d x,
\end{align*}
with the test function $\fy^n=\frac{q^\dag(t_n)-q_h^{n*}}{q^\dag(t_n)} u(t_n) \in H_0^1(\Omega).$
By the box constraint of $\Uad$, the assumption $|\nabla q^\dag| \le c$ and the regularity
estimate $\| u(t) \|_{H^2\II} \le c $ from \eqref{reg-fde-2},  we have
\begin{equation}\label{eqn:nablafyn}
  \|\nabla\fy^n\|_{L^2(\Omega)}\le c(1+\|\nabla q_h^{n*}\|_{L^2(\Omega)}),
\end{equation}
Meanwhile, by integration by parts, we have the splitting
 \begin{align*}
 ((q^\dag(t_n)-q_h^{n*})\nabla u(t_n),\nabla\fy^n)&=-(\nabla\cdot((q^\dag(t_n)-q_h^{n*})\nabla u(t_n)), \fy^n-P_h\fy^n)\nonumber\\
    &\quad + (q_h^{n*}\nabla (U_h^n(q_{h,\tau}^*) - u(t_n)),\nabla P_h\fy^n) \label{eqn:sp-01} \\
   & \quad + (q^\dag(t_n)\nabla u(t_n) - q_h^{n*}\nabla U_h^n(q_{h,\tau}^*),\nabla P_h\fy^n) =\sum_{i=1}^3{\rm I}_i^n.\nonumber
\end{align*}
Below we bound the terms separately.\\
\noindent\textbf{Step 1: bound the term ${\rm I}_1^n$.}
Since $q^\dag,q_{h\tau}^*\in \Uad$, $|\nabla q^\dag| \le c$,
and  $\| u(t) \|_{H^2\II} \le c $ from \eqref{reg-fde-2}
and $\| \nabla u(t) \|_{L^\infty\II} \le c $ from \eqref{reg-fde-inf}, we derive
\begin{align*}
  \| \nabla\cdot((q^\dag(t_n)-q_h^{n*})\nabla u(t_n))\|_{L^2(\Omega)}
 \le & \| \nabla q^\dag(t_n)\|_{L^\infty(\Omega)}  \| \nabla u(t_n) \|_{L^2(\Omega)} +\| q^\dag(t_n)-q_h^{n*}\|_{L^\infty(\Omega)}\| \Delta u(t_n) \|_{L^2(\Omega)}\\
   & +  \| \nabla q_h^{n*}\|_{L^2(\Omega)}  \| \nabla u(t_n) \|_{L^\infty(\Omega)}\le c(1 +  \| \nabla q_h^{n*}\|_{L^2(\Omega)}).
\end{align*}
Then the Cauchy-Schwarz inequality and the approximation property of $P_h$  in \eqref{eqn:proj-L2-error} imply
\begin{equation*}
 |{\rm I}_1^n|  \le  c(1+\| \nabla q_h ^{n*}\|_{L^2(\Omega)} ) \|  \fy^n-P_h\fy^n\|_{L^2(\Omega)}\le c h(1+\| \nabla q_h ^{n*}\|_{L^2(\Omega)} ) \| \nabla \fy^n  \|_{L^2(\Omega)}.
\end{equation*}
Thus, we can bound the term ${\rm I}_1^n$ by
\begin{align}
 \tau \sum_{n=1}^N   |{\rm I}_1^n|  \le  c h \tau\sum_{n=1}^N(1+\| \nabla q_h ^{n*}\|_{L^2(\Omega)} )^2
\le & ch + c h  \|(\nabla q_h ^{n*})_{n=1}^N\|_{\ell^2(L^2(\Omega))}  ^2  \le c ( h + h \gamma^{-1}\eta^2).\label{eqn:I1}
\end{align}
\noindent\textbf{Step 2: bound the term ${\rm I}_2^n$.}
For the term ${\rm I}_2^n$, by the triangle inequality, inverse inequality for
functions in $X_h$, the $L^2(\Omega)$ stability of $P_h$ in \eqref{eqn:proj-L2-error}, we deduce
\begin{align*}
 \|  \nabla(u(t_n) - U_h^n(q_{h,\tau}^{*})) \|_{L^2(\Omega)}
  & \leq \|  \nabla(u(t_n) - P_hu(t_n) ) \|_{L^2(\Omega)} + h^{-1}\|  P_h u(t_n)  - U_h^n(q_{h,\tau}^*)  \|_{L^2(\Omega)}\\
  & \leq c(h + h^{-1}\|P_hu(t_n) - U_h^n(q_{h,\tau}^*)   \|_{L^2(\Omega)}).
\end{align*}
Meanwhile, by the standard energy argument \cite[Lemma 3.6]{JinZhou:2021sicon}, we deduce
\begin{equation*}
  \|(\nabla U_h^n(q_h^*))_{n=1}^N\|_{\ell^2(L^2(\Omega))} ^2\leq c(\|(f(t_n))_{n=1}^N\|_{\ell^2(H^{-1}(\Omega))}^2+\|\nabla u_0\|_{L^2(\Omega)}^2)\leq c.
\end{equation*}
This and the regularity estimate \eqref{reg-fde-2} imply
$$\|(\nabla(u(t_n)-U_h^n(q_{h,\tau}^*)))_{n=1}^N\|_{\ell^2(L^2(\Omega))}^2\leq c.$$
Thus, the Cauchy-Schwarz inequality, Lemma \ref{lem:err-2}
and \eqref{eqn:nablafyn} imply
\begin{equation*}\label{eqn:I2}
\begin{split}
&\quad\tau \sum_{n=1}^N   |{\rm I}_2^n|
 \le  c\tau  \sum_{n=1}^N \|  \nabla(u(t_n) - U_h^n(q_{h,\tau}^*)) \|_{L^2(\Omega)} \|  \nabla \fy^n \|_{L^2(\Omega)}\\
& \le  c\|  (\nabla(u(t_n) - U_h^n(q_{h,\tau}^*)))_{n=1}^N \|_{\ell^2(L^2(\Omega))} \|  (\nabla \fy^n)_{n=1}^N \|_{\ell^2(L^2(\Omega))}\\
  &\le c \min\big(1,h+h^{-1}  \| ( u(t_n) - U_h^n(q_{h,\tau}^*))_{n=1}^N  \|_{\ell^2(L^2(\Omega))}\big) (1+\|(\nabla q_{h}^{n*})_{n=1}^N \|_{\ell^2(L^2(\Omega))})\\
  &\le  c \min(1,h+h^{-1} \eta)\gamma^{-\frac12} \eta.
\end{split}
\end{equation*}
\noindent\textbf{Step 3: bound the term ${\rm I}_3^n$.} The estimate of the term ${\rm I}_3^n$ is more technical.
It follows directly from the weak formulations \eqref{eqn:var} and \eqref{eqn:fully} that
\begin{align*}
 {\rm I}_3^n 
 &=(\bar \partial_\tau^\alpha [(U_h^n(q_{h,\tau}^*) - U_h^0)-  (u(t_n)-u_0)],  P_h\fy^n )\\
  &\quad+ ( \bar \partial_\tau^\alpha (u(t_n) - u_0) - \partial_t^\alpha (u(t_n) - u_0) , P_h\fy^n )=:  {\rm I}_{3,1}^n + {\rm I}_{3,2}^n.
\end{align*}
Next we bound the two terms ${\rm I}_{3,1}^n$ and ${\rm I}_{3,2}^n$ separately.
By Lemma \ref{lem:deriv-approx}, there holds
\begin{equation*}
|{\rm  I}_{3,2}^n| \le  \| \bar \partial_\tau^\alpha u(t_n)  - \partial_t^\alpha u(t_n)  \|_{L^2(\Omega)}  \| P_h\fy^n \|_{L^2(\Omega)}
           \le c\tau (t_n^{-1} + \ell_n), \quad n=1,2,\ldots, N.
\end{equation*}
Consequently,
\begin{equation*}
 |\tau^2\sum_{m=1}^N \sum_{n=1}^m {\rm I}_{3,2}^n| \le  c \tau^3  \sum_{m=1}^N \sum_{n=1}^m (t_n^{-1} + \ell_n) \ell_n \le c \tau \ell_N .
\end{equation*}
Since $U_h^0(q_{h,\tau}^*)=P_hu_0$ and $u(0)=u_0$, by summation by parts, we have
\begin{align*}
\tau \sum_{n=1}^m {\rm I}_{3,1}^n&= \tau \sum_{n=0}^m  (\bar \partial_\tau^\alpha (U_h^n(q_{h,\tau}^*)-u(t_n)),  P_h\fy^n)
=  \tau \sum_{j=0}^m  ( U_h^j(q_{h,\tau}^*)-u(t_j),  \tau^{-\alpha} \sum_{n=j}^m b_{n-j}^{(\alpha)} P_h\fy^n ).
\end{align*}
Next we appeal to the splitting
\begin{equation*}
  \tau^{-\alpha} \sum_{n=j}^m b_{n-j}^{(\alpha)} P_h\fy^n = \tau^{-\alpha} \sum_{n=j}^m b_{n-j}^{(\alpha)} P_h(\fy^n-\fy^m) + \tau^{-\alpha} \sum_{n=j}^m b_{n-j}^{(\alpha)} P_h\fy^m
: = {\rm II}_{j,m}^1 +  {\rm II}_{j,m}^2.
\end{equation*}
For the weights $b_n^{(\alpha)}$, we have
$\sum_{n=0}^mb_n^{(\alpha)} = b_m^{(\alpha-1)}$ and
$|\tau^{-\alpha} \sum_{n=0}^{m} b_{n}^{(\alpha)}| \leq ct_{m+1}^{-\alpha}$ \cite[Exercise 6.16]{Jin:2021book}.
In view of this and the estimate $\|\fy^m\|_{L^2(\Omega)}\leq c$, the sum ${\rm II}_{j,m}^2$ satisfies
\begin{align*}
 \| {\rm II}_{j,m}^2\|_{L^2(\Omega)} & \le c \| \fy^m \|_{L^2(\Omega)} \Big(\tau^{-\alpha} \sum_{n=0}^{m-j} b_{n}^{(\alpha)}\Big)
   \le c t_{m-j+1}^{-\alpha} \| \fy^m \|_{L^2(\Omega)}\le c t_{m-j+1}^{-\alpha}.
\end{align*}
Then Lemma \ref{lem:err-2}, the Cauchy-Schwarz inequality and Young's inequality for (discrete) convolution imply
\begin{align*}
 \tau^2 \sum_{m=1}^N \sum_{j=1}^m \|U_h^j(q_{h,\tau}^*)  -  u(t_j)  \|_{L^2(\Omega)}  \|{\rm II}_{j,m}^2\|_{L^2(\Omega)}
 \le& c \tau^2 \sum_{j=1}^N \sum_{m=j}^N \|U_h^j(q_{h,\tau}^*)  -  u(t_j)  \|_{L^2(\Omega)}  t_{m-j+1}^{-\alpha}\\
 \le&  c\|(U_h^j(q_{h,\tau}^*)  -  u(t_j))_{j=1}^N \|_{\ell^2(L^2(\Omega))}
 \le c\eta,
\end{align*}
Similarly, by Lemma \ref{lem:CQ} and the Cauchy-Schwarz inequality, we have
\begin{equation*}
\begin{aligned}
&\tau^2 \sum_{m=1}^N \sum_{j=1}^m \|U_h^j(q_{h,\tau}^*)-u(t_j)\|_{L^2(\Omega)}\|{\rm II}_{j,m}^1\|_{L^2(\Omega)}\\
\le& c\tau \sum_{m=1}^N \|(U_h^j(q_{h,\tau}^*)-u(t_j))_{j=1}^m\|_{\ell^2(L^2(\Omega))} \|({\rm II}_{j,m}^1)_{j=1}^m\|_{\ell^2(L^2(\Omega))}  \\
\le&   c \gamma^{-\frac12} \eta   \|(U_h^j(q_{h,\tau}^*)-u(t_j))_{j=1}^N\|_{\ell^2(L^2(\Omega))}\le c  \gamma^{-\frac12} \eta^2.
\end{aligned}
\end{equation*}
These two estimates and the triangle inequality lead to
\begin{equation}\label{eqn:I3}
\Big|\tau^2 \sum_{m=1}^N \sum_{n=1}^m(\bar\partial_\tau^\alpha (U_h^n(q_{h,\tau}^*) - u(t_n)),P_h\fy^n )\Big|\le c\eta +  c\gamma^{-\frac12} \eta^2 .
\end{equation}
The three estimates \eqref{eqn:I2}, \eqref{eqn:I1}, and \eqref{eqn:I3} together imply
\begin{align*}
 \Big|\tau^2\sum_{m=1}^N \sum_{n=1}^m((q^\dag-q_{h}^{n*})\nabla u(t_n),\nabla\fy^n)\Big|
  \le  c(h \gamma^{-1}\eta + \gamma^{-\frac12}\eta+ h^{-1}\gamma^{-\frac12}\eta+\gamma^{-\frac12}\eta)\eta.
\end{align*}
Combining the preceding estimates gives the desired error estimate.
\qed

{\begin{remark}\label{rmk:error-q}
Under the $\beta$-positivity condition in Definition \ref{def:condP1}, for any $\delta>0$,
with $\eta= \tau + h^2  + \delta + \gamma^{\frac12}$, the argument of Theorem \ref{thm:stab1} gives
\begin{equation*}
  \|(q^\dag(t_n)-q_{h}^{n*})_{n=1}^N\|_{\ell^2(L^2(\Omega))}\leq c ((h \gamma^{-1}\eta+ \gamma^{-\frac12}\min(1,h^{-1}\eta))\eta)^\frac{1}{2(1+\beta)}.
\end{equation*}
Theorem \ref{thm:error-q} provides useful guidelines for choosing the regularization
parameter $\gamma$ and the discretization parameters $h$ and $\tau$. Indeed, by
suitably balancing the terms in the estimate, we should choose $\gamma \sim \delta^2$,
$h\sim\sqrt{\delta}$ and $\tau\sim\delta$ in practical computation in order to
effect optimal computational complexity. Under the $\beta$-positivity
condition, this choice of $\gamma,h$ and $\tau$ gives
$$\|(q^\dag(t_n)-q_{h}^{n*})_{n=1}^N\|_{\ell^2(L^2(\Omega))}  \le c\delta^{\frac1{4(1+\beta)}}.$$
Note that this result is consistent with Theorem \ref{thm:stab2}.
\end{remark}}

\section{Numerical results and discussions}\label{sec:numer}

Now we present numerical experiments to illustrate the
feasibility of recovering a space-time dependent diffusion
coefficient $q^\dag(x,t)$. Throughout, the corresponding discrete optimization problem is
solved by the conjugate gradient (CG) method (cf. \cite{Alifanov:1995}), with the
gradient computed using the standard adjoint technique. The lower and upper
bounds in the admissible set $\mathcal{A}$ are taken to be
$c_0=0.5$ and $c_1=5$, respectively, and are enforced by a projection step after
each CG iteration. Generally, the algorithm converges within tens of
iterations, with the maximum number of iterations fixed at 100. The noisy data $z^\delta$ is generated by
\begin{equation*}
  z^\delta(x,t) = u(q^\dag)(x,t) + \epsilon\|u(q^\dag)\|_{L^\infty(0,T;L^\infty(\Omega))}\xi(x,t),\quad (x,t)\in \Omega\times(0,T),
\end{equation*}
where $\xi(x,t)$ follows the standard Gaussian distribution, and $\epsilon\geq 0$ denotes
the (relative) noise level. The reference data $u(q^\dag)$ is computed with
a finer mesh. The noisy data $z^\delta$ is first generated on a fine
spatial-temporal mesh and then interpolated to a coarse spatial/ temporal mesh for the
inversion step. The regularization parameter $\gamma$ in the functional $J_\gamma$
is determined in a trial and error manner.

\subsection{Numerical results in one spatial dimension}
First we present numerical results for two examples on unit interval $\Omega=(0,1)$.
The first example has a smooth exact coefficient $q^\dag$, and the problem is homogeneous.
\begin{example}\label{exam:1dsmooth}
$u_0=x(1-x)$, $f\equiv 0$, $q^\dag=2+\sin(\pi x)e^{-0.1t}$, $T=0.1$.
\end{example}

The numerical results for Example \ref{exam:1dsmooth} with different level $\epsilon$ of noises are
shown in Table \ref{tab:exam1}, where the quantities $e_q$ and $e_u$, respectively, defined by
\begin{equation*}
e_q = \|(q_{h}^{n*}-q^\dag(t_n))_{n=1}^N\|_{\ell^2(L^2(\Omega))}\quad\mbox{and}\quad e_q = \|(U_{h}^n(q_{h,\tau}^*)-u(q^\dag)(t_n))_{n=1}^N\|_{\ell^2(L^2(\Omega))}
\end{equation*}
are used to measure the convergence of the discrete approximations. The results are
computed with a fixed small time step size $\tau=1\times 10^{-4}$, and $\gamma\sim O(\delta^2)$
and $h\sim O(\sqrt{\delta})$, cf. Remark \ref{rmk:error-q}. It is observed that the
$\ell^2(L^2(\Omega))$ error $e_q$ of the reconstruction $q_{h,\tau}^*$ decreases
steadily as the noise level $\epsilon$ tends to zero with a rate roughly $O(\delta^{0.40})$.
This convergence rate is consistently observed for all three fractional orders, and thus
the order $\alpha$ does not influence much the convergence rates, provided that the time step
size $\tau$ is sufficiently small. The empirical rate is faster than the theoretical one in Theorem \ref{thm:error-q}.
It remains an outstanding question to obtain the optimal convergence of discrete approximations. Meanwhile,
the quantity $e_u$ converges also to zero as the noise level $\epsilon\to 0$, at a rate
nearly $O(\delta^{0.9})$, which agrees well with the theoretical prediction $O(\delta)$
from Lemma \ref{lem:err-2}. We refer to Fig. \ref{fig:recon-exam1} for exemplary reconstructions:
the recoveries are qualitatively comparable with each other and all reasonably accurate
for both $\epsilon=\text{1.00e-2}$ and $\epsilon=\text{5.00e-2}$, thereby concurring
with the errors in Table \ref{tab:exam1}.

\begin{table}[hbt!]
\setlength{\tabcolsep}{8pt}
\centering
\caption{The errors $e_q$
and $e_u$ for Example \ref{exam:1dsmooth}.\label{tab:exam1}}
\begin{tabular}{c|c|c|cccccc}
\hline
         & $\epsilon$ & 5.00e-2  &  3.00e-2  &  1.00e-2  &  5.00e-3  &  3.00e-3  &  1.00e-3  &     \\
$\alpha$ & $\gamma$   & 5.00e-10 &  1.80e-10 &  2.00e-11 &  5.00e-12 &  1.80e-12 &  2.00e-13 & rate\\
\hline
$0.25$ & $e_q$        & 1.26e-2  &  1.28e-2  &  5.57e-3  &  4.00e-3  &  3.27e-3  &  2.45e-3  &0.467\\
       & $e_u$        & 1.65e-5  &  1.14e-5  &  5.25e-6  &  3.31e-6  &  1.65e-6  &  4.84e-7  &0.880\\
$0.50$ & $e_q$        & 1.07e-2  &  1.47e-2  &  6.86e-3  &  5.15e-3  &  4.04e-3  &  3.28e-3  &0.375\\
       & $e_u$        & 3.93e-5  &  2.83e-5  &  1.48e-5  &  6.80e-6  &  3.41e-6  &  1.22e-6  &0.897\\
$0.75$ & $e_q$        & 1.01e-2  &  9.09e-3  &  7.06e-3  &  4.77e-3  &  3.93e-3  &  2.50e-3  &0.363\\
       & $e_u$        & 6.40e-5  &  2.71e-5  &  1.70e-5  &  5.69e-6  &  4.37e-6  &  1.57e-6  &0.916\\
\hline
\end{tabular}
\end{table}

\begin{figure}[hbt!]
\setlength{\tabcolsep}{0pt}
   \centering
   \begin{tabular}{ccccc}
   \includegraphics[width=0.2\textwidth]{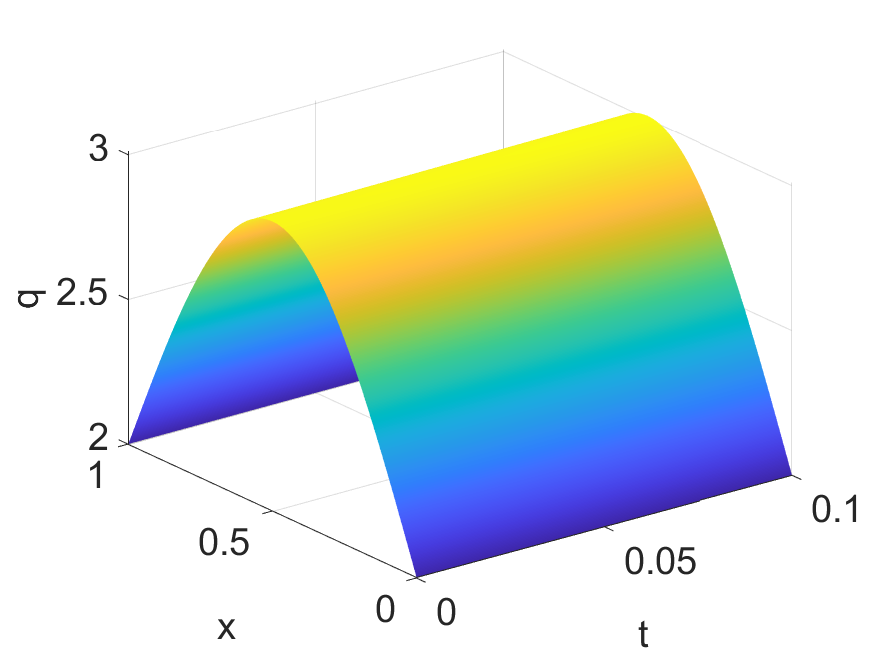}&\includegraphics[width=0.2\textwidth]{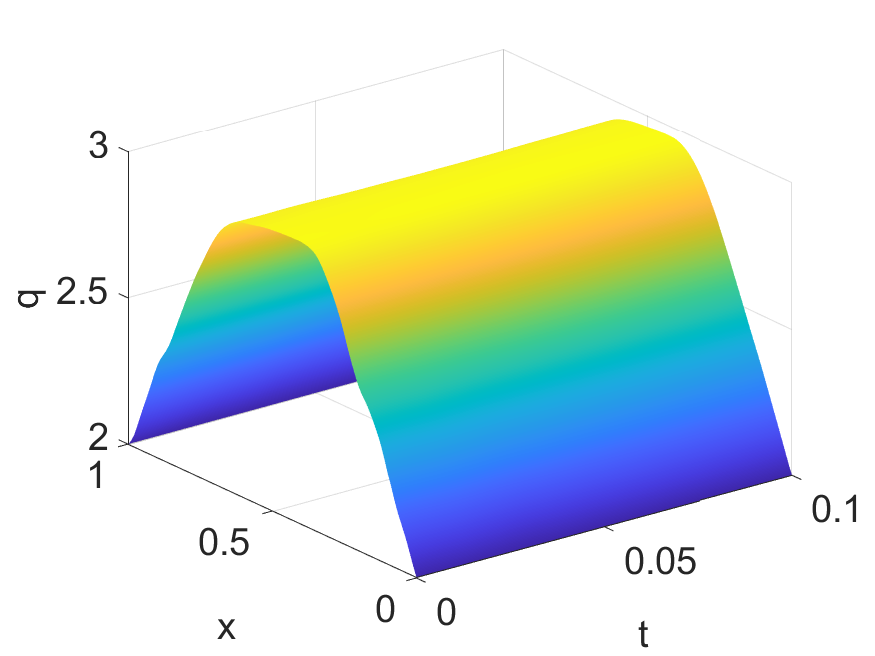}
   & \includegraphics[width=0.2\textwidth]{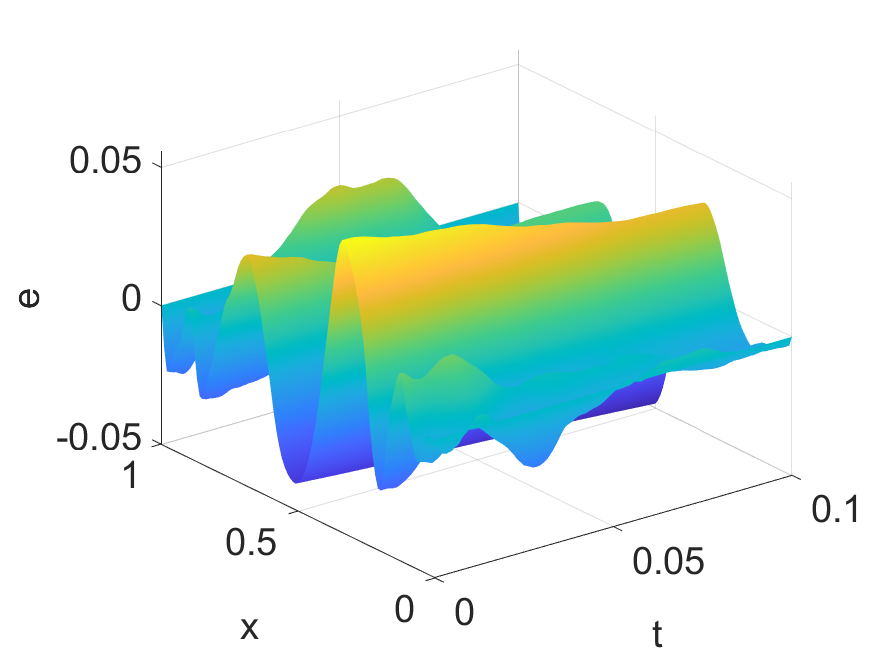} & \includegraphics[width=0.2\textwidth]{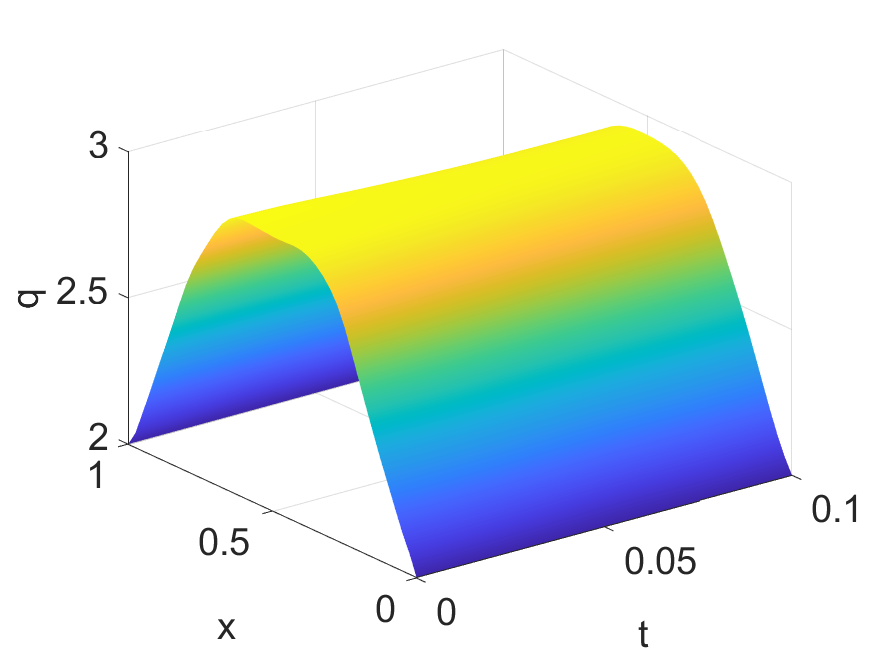}
       & \includegraphics[width=0.2\textwidth]{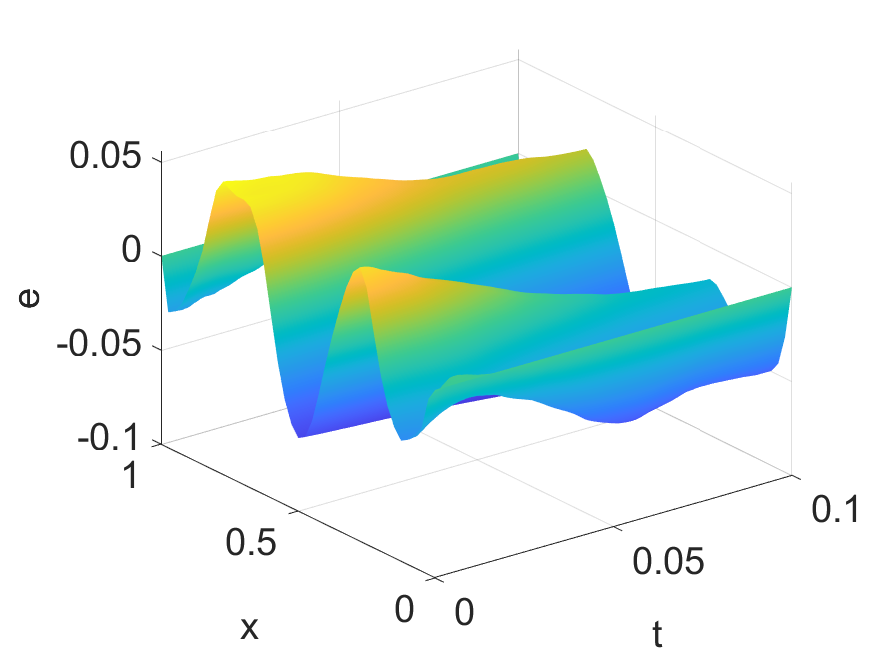}\\
      $q^\dag$  & $\epsilon=\text{1.00e-2}$ & & $\epsilon = \text{5.00e-2}$
   \end{tabular}
   \caption{Numerical reconstructions and the corresponding pointwise error $e=q_{h,\tau}^*-q^\dag$ for Example \ref{exam:1dsmooth} with $\alpha=0.5$, at two noise levels $\epsilon = $ 1.00e-2 and 5.00e-2.
   \label{fig:recon-exam1}}
\end{figure}

The second example has a nonsmooth coefficient $q^\dag$.
\begin{example}\label{exam:1dnonsmooth}
$u_0(x)=x(1-x)$, $f\equiv 0 $, $q^\dag=2+\min(x,1-x)(1-t)$, $T=0.1$.
\end{example}

The numerical results for Example \ref{exam:1dnonsmooth} with different levels of noise
are given in Table \ref{tab:exam2}. Note that the exact coefficient $q^\dag$ does not
satisfy the regularity condition in Assumption \ref{ass:zdelta}, and thus one expects
the convergence rates of $e_q$ and $e_u$ suffer from a loss. Indeed, the error
$e_q$ converges at a slower rate $O(\delta^{0.3})$, which, however, is still higher than that
predicted by Remark \ref{rmk:error-q}. Interestingly, the error $e_u$ converges
roughly at the rate $O(\delta)$, confirming the estimate in Lemma \ref{lem:err-2}.
This observation holds for all three fractional orders.
Exemplary reconstructions are shown in Fig. \ref{fig:recon-exam2}, which shows clearly
the convergence of the discrete approximations as the noise level $\epsilon$ decreases.

\begin{table}[hbt!]
\setlength{\tabcolsep}{8pt}
\centering
\caption{The errors $e_q$ and $e_u$ for Example \ref{exam:1dnonsmooth}.\label{tab:exam2}}
\begin{tabular}{c|c|c|cccccc}
\hline
         & $\epsilon$ & 5.00e-2 &  3.00e-2  & 1.00e-2  & 5.00e-3  & 3.00e-3  & 1.00e-3  &     \\
$\alpha$ & $\gamma$   & 1.00e-9 &  3.60e-10 & 4.00e-11 & 1.00e-11 & 3.60e-12 & 4.00e-13 & rate\\
\hline
$0.25$ & $e_q$        & 9.58e-3 &  7.59e-3 &  5.77e-3 &  5.10e-3 &  4.55e-3 &  3.71e-3  &0.234\\
       & $e_u$        & 1.85e-5 &  1.15e-5 &  5.02e-6 &  3.14e-6 &  1.51e-6 &  4.82e-7  &0.910\\
$0.50$ & $e_q$        & 1.28e-2 &  8.17e-3 &  6.39e-3 &  4.70e-3 &  4.11e-3 &  3.94e-3  &0.297\\
       & $e_u$        & 5.44e-5 &  2.88e-5 &  1.05e-5 &  7.79e-6 &  3.39e-6 &  1.02e-6  &0.977\\
$0.75$ & $e_q$        & 1.17e-2 &  8.41e-3 &  6.02e-3 &  4.07e-3 &  4.05e-3 &  3.75e-3  &0.301\\
       & $e_u$        & 5.93e-5 &  3.32e-5 &  1.44e-5 &  7.31e-6 &  4.14e-6 &  1.33e-6  &0.951\\
\hline
\end{tabular}
\end{table}

\begin{figure}[hbt!]
\setlength{\tabcolsep}{0pt}
   \centering
   \begin{tabular}{ccccc}
   \includegraphics[width=0.2\textwidth]{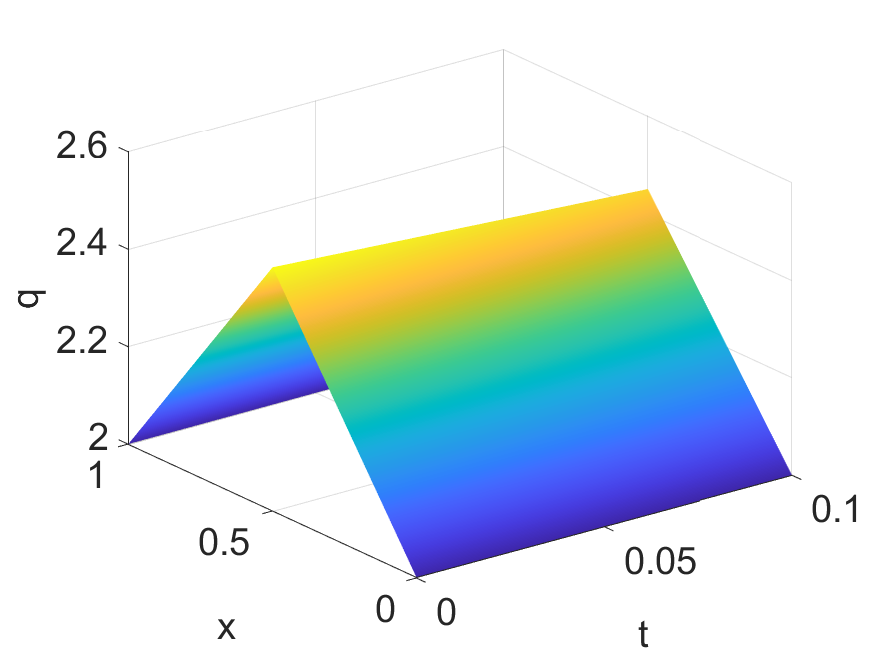}&\includegraphics[width=0.2\textwidth]{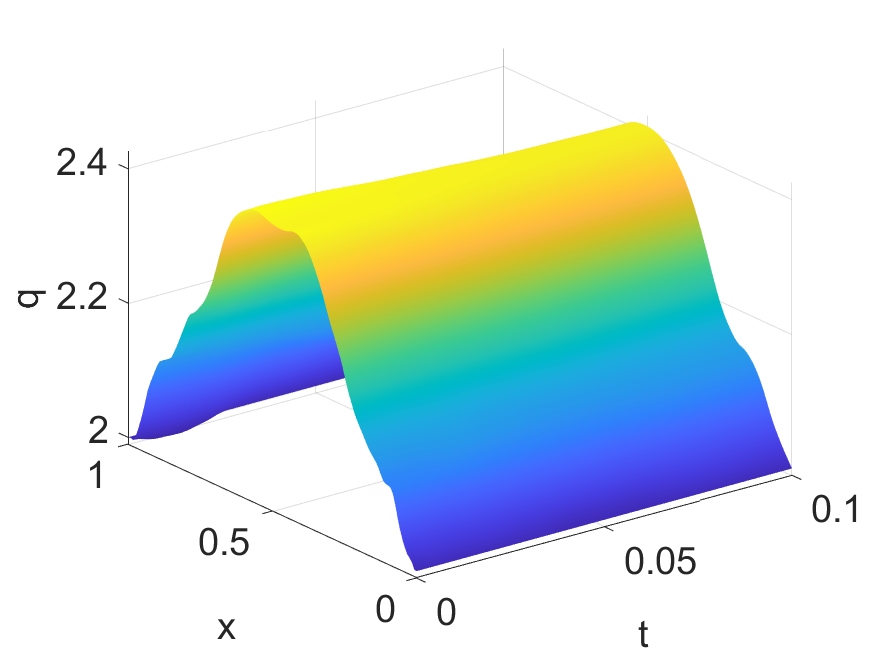} & \includegraphics[width=0.2\textwidth]{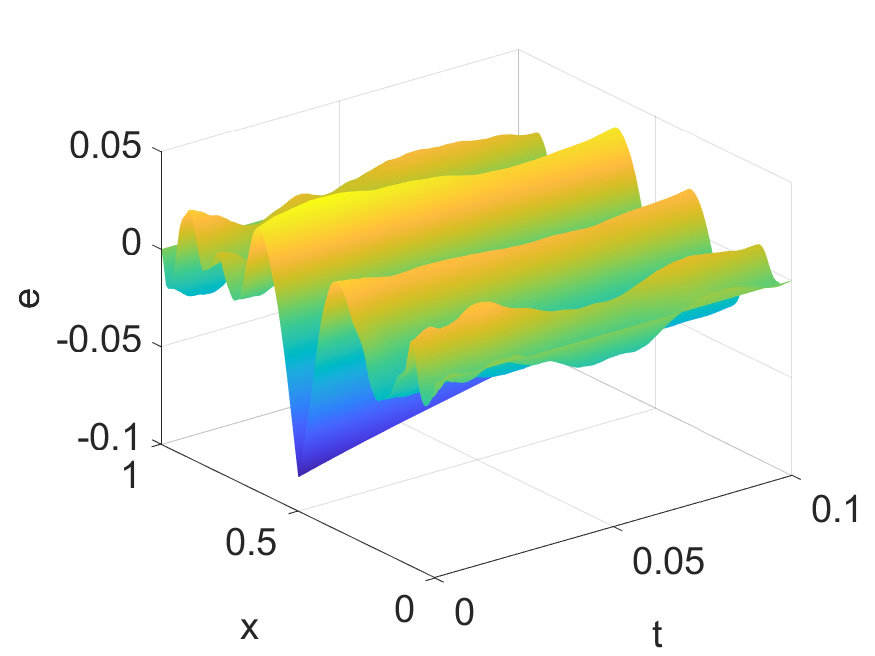} & \includegraphics[width=0.2\textwidth]{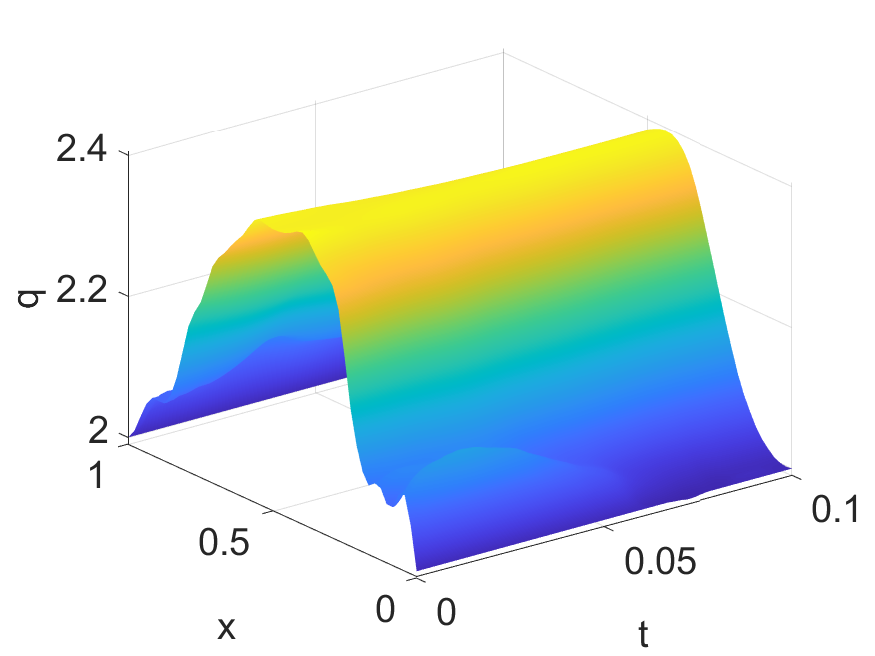}& \includegraphics[width=0.2\textwidth]{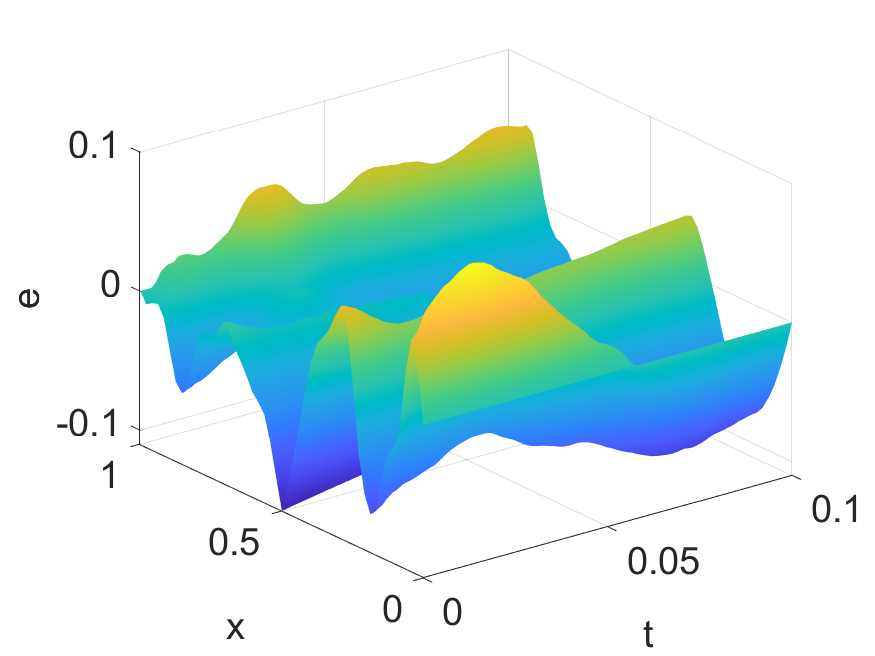}\\
      $q^\dag$  & $\epsilon=\text{1.00e-2}$ & & $\epsilon = \text{5.00e-2}$
   \end{tabular}
   \caption{Numerical reconstructions and the corresponding pointwise error $e=q_{h,\tau}^*-q^\dag$ for Example \ref{exam:1dnonsmooth} with $\alpha=0.5$, at two noise levels $\epsilon=$ 1.00e-2 and 5.00e-2.
   \label{fig:recon-exam2}}
\end{figure}

\subsection{Numerical results in two spatial dimension}

Now we present numerical results for the following example on the unit square $\Omega=(0,1)^2$.
The domain $\Omega$ is first uniformly divided into $M^2$ small squares, each with side length $1/M$, and then a
uniform triangulation is obtained by connecting the low-left and upper-right vertices of each small square.
The reference data is first computed on a finer mesh with $M=100$ and a time step size $\tau=1/2000$.
The inversion step is carried out with a mesh $M=40$ and $\tau=1/500$.

\begin{example}\label{exam:2d}
$u_0(x_1,x_2)=x_1(1-x_1)\sin (\pi x_2)$, $f=\sin(\pi x_1)\sin(\pi x_2)(1+t)$, $q^\dag(x_1,x_2)=1+\sin(\pi x_1) x_2(1-x_2)$, and $T=1$.
\end{example}

The numerical results for the example with different noise levels are presented in Fig.
\ref{fig:recon-exam3}. The empirical observations are in excellent agreement
with that for the one-dimensional problem in Example \ref{exam:1dsmooth}: we observe a steady
convergence as the noise level $\epsilon$ decreases to zero. The plots also indicate that
for the pointwise error $e=q_h^*-q^\dag$, the error in recovering the peak is dominating,
however, the overall shape is well resolved.

\begin{figure}[hbt!]
\setlength{\tabcolsep}{0pt}
   \centering
   \begin{tabular}{ccccc}
   \includegraphics[width=0.20\textwidth]{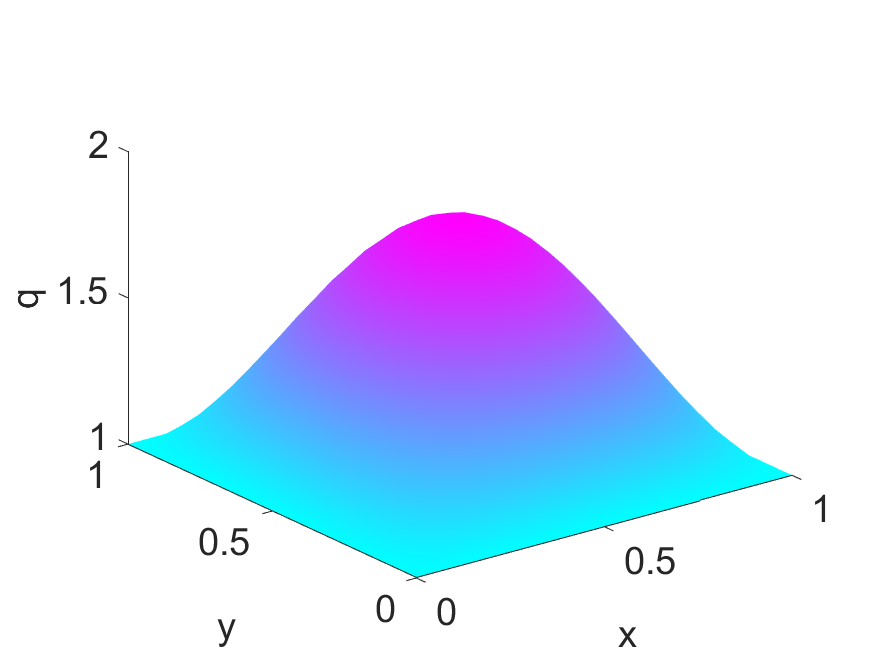}&\includegraphics[width=0.20\textwidth]{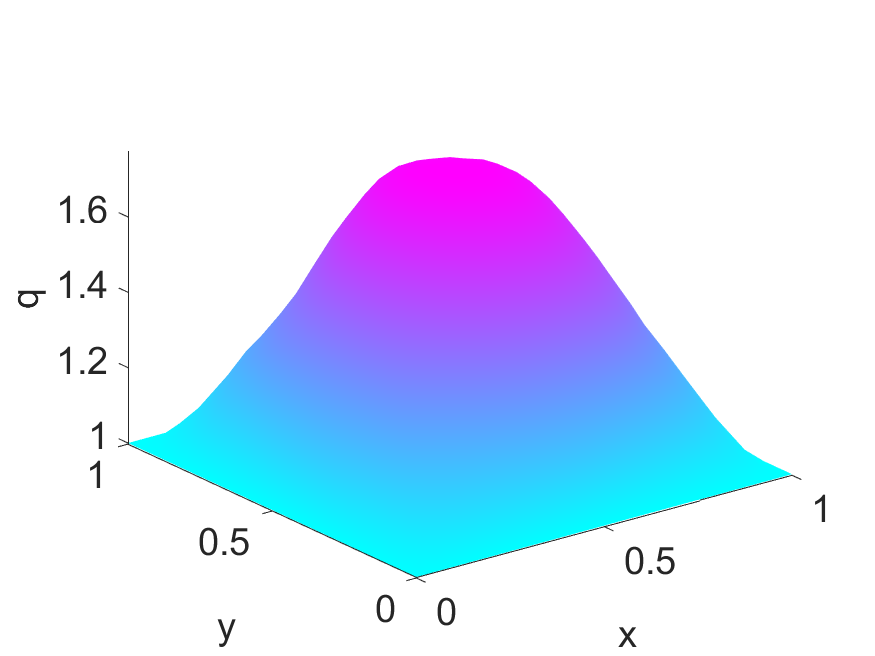} & \includegraphics[width=0.20\textwidth]{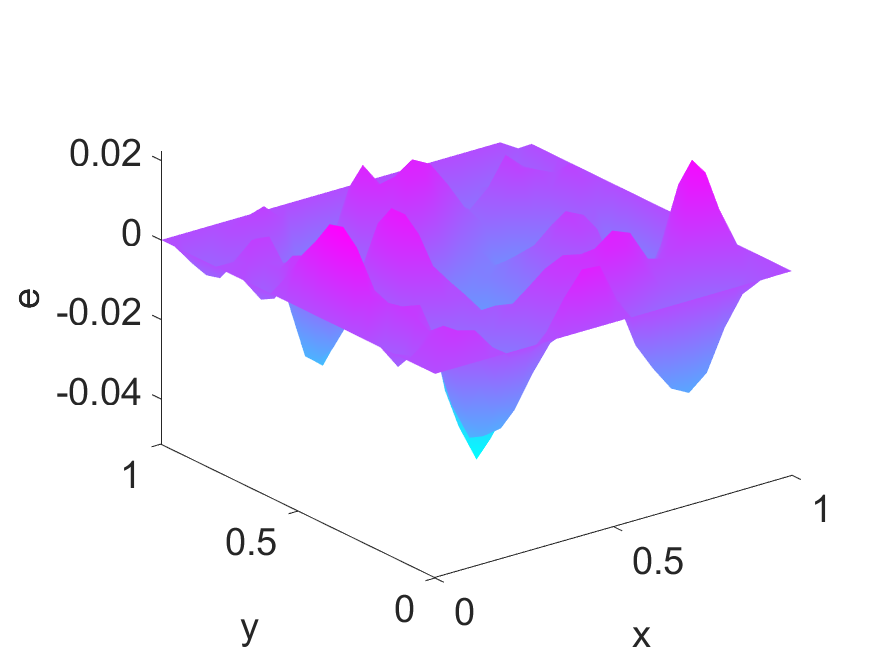} & \includegraphics[width=0.20\textwidth]{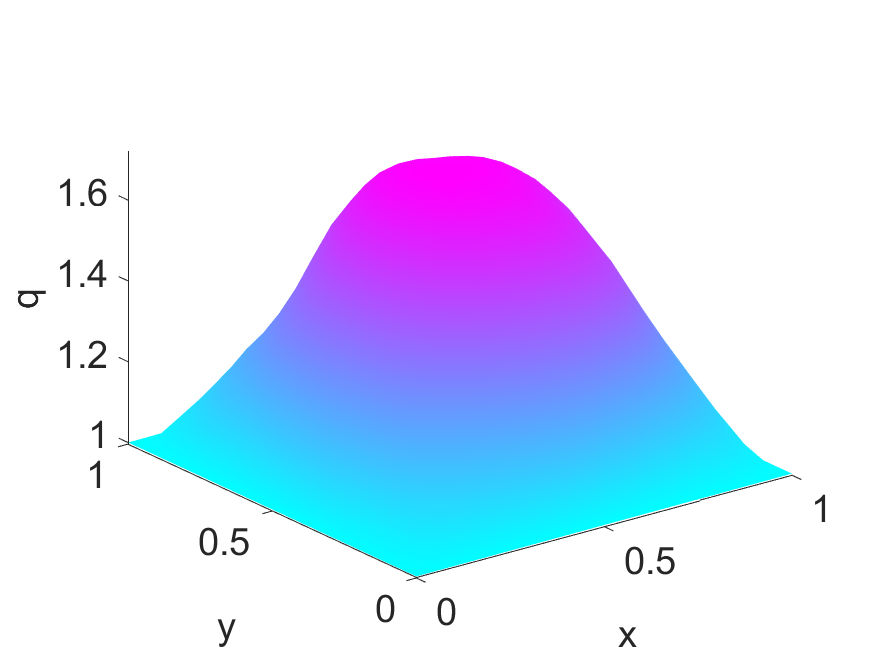}
      &  \includegraphics[width=0.20\textwidth]{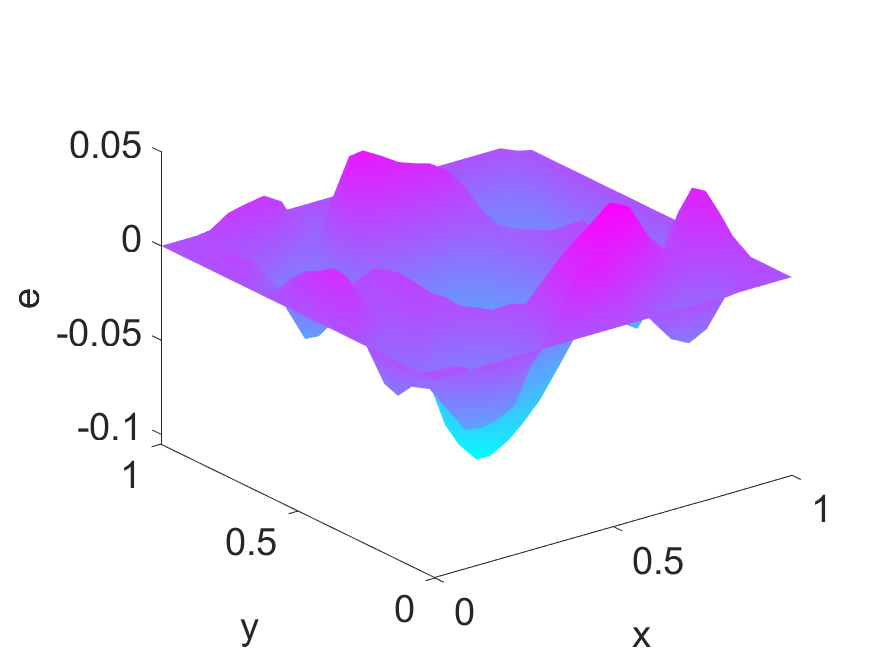}\\
   &\includegraphics[width=0.20\textwidth]{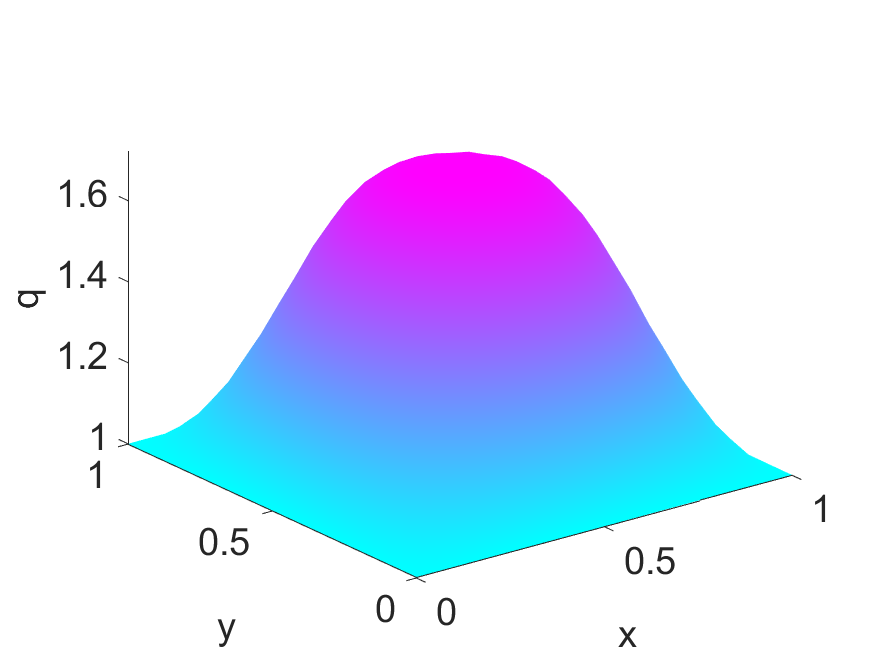}& \includegraphics[width=0.2\textwidth]{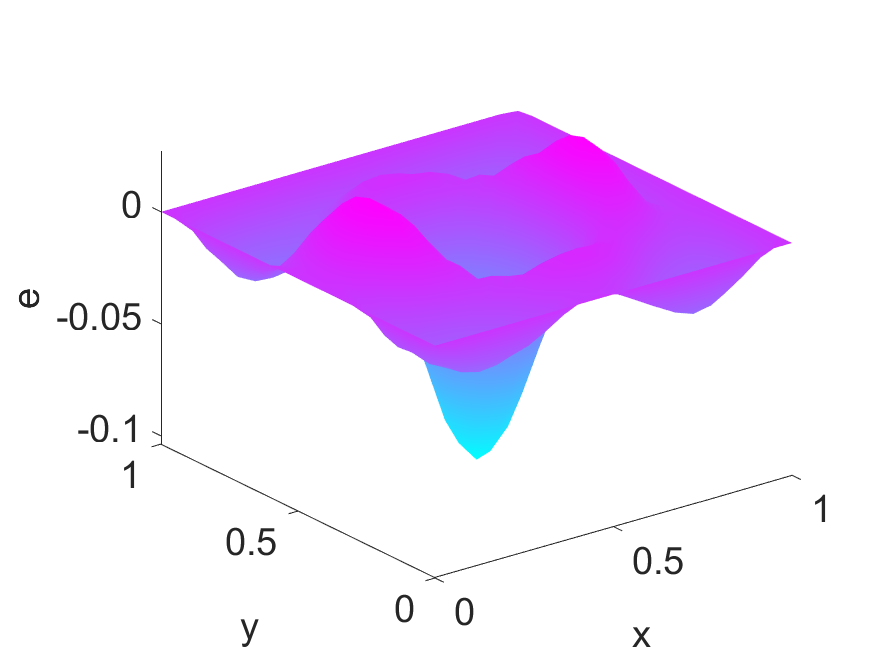} & \includegraphics[width=0.20\textwidth]{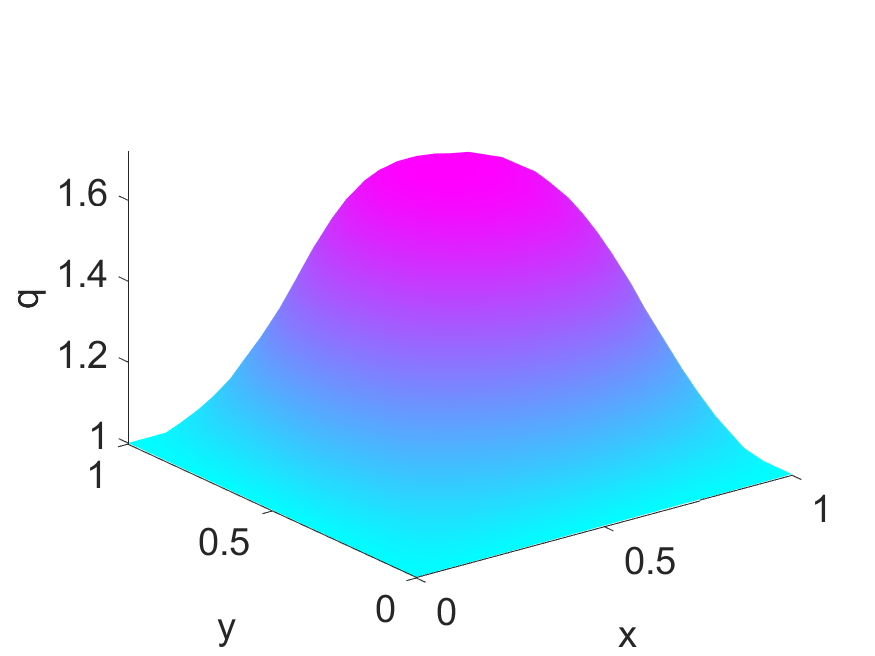} & \includegraphics[width=0.2\textwidth]{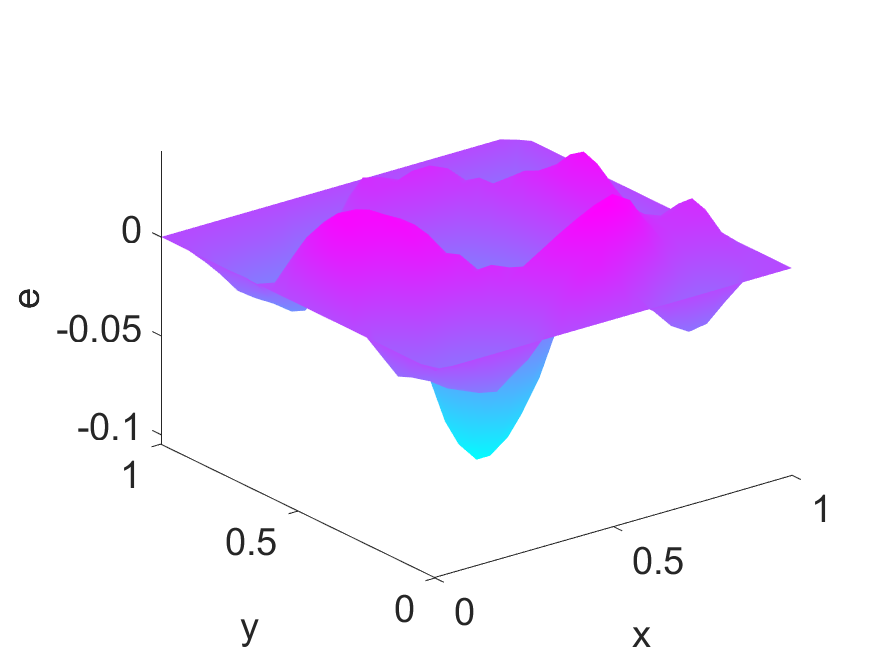}\\
   &\includegraphics[width=0.20\textwidth]{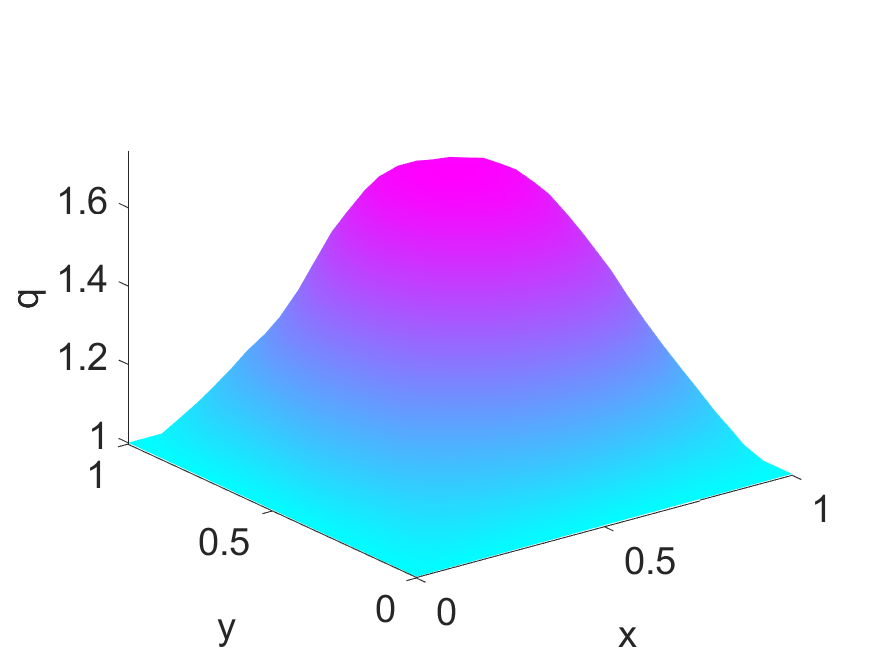}
   & \includegraphics[width=0.2\textwidth]{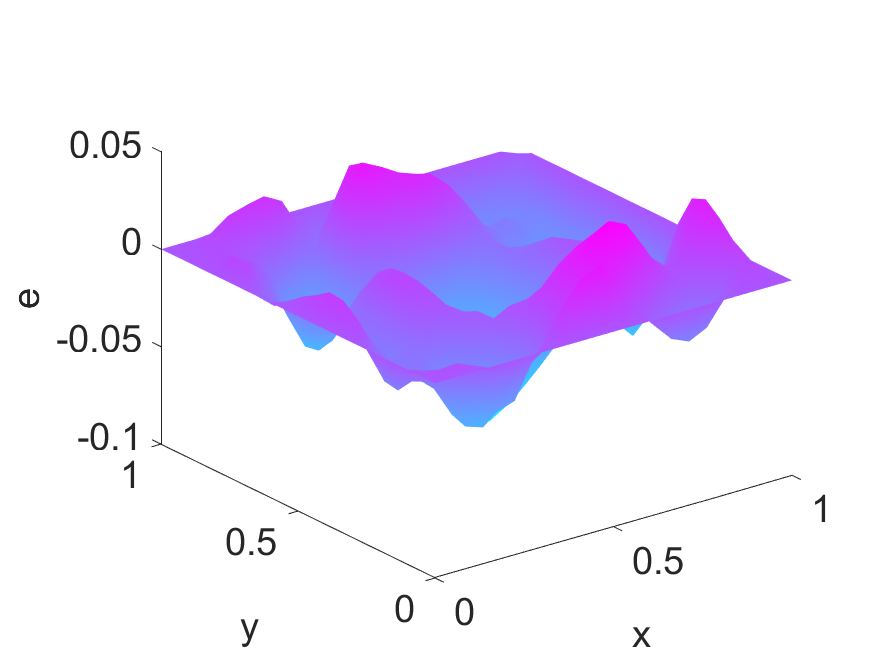} & \includegraphics[width=0.2\textwidth]{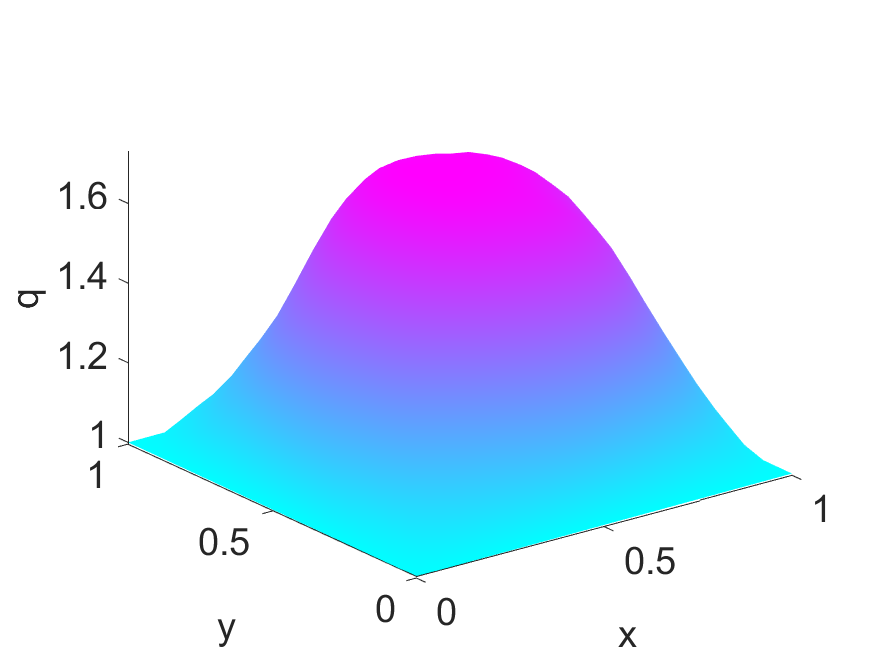}
       & \includegraphics[width=0.2\textwidth]{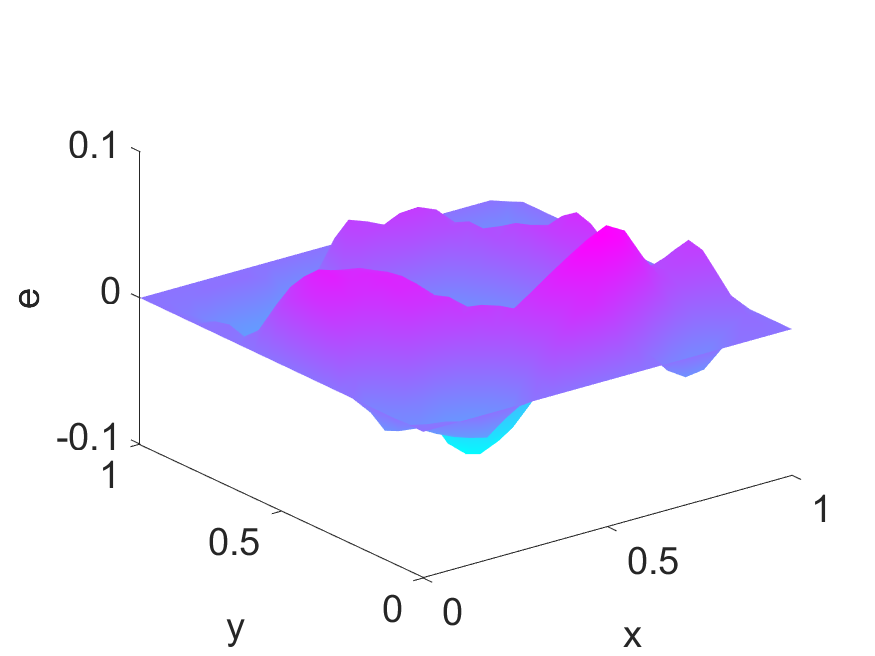}\\
      $q^\dag$  & $\epsilon=\text{1.00e-2}$ & & $\epsilon = \text{5.00e-2}$\\
   \end{tabular}
   \caption{Numerical reconstructions for Example \ref{exam:2d} with $\epsilon = $1e-2 and 5.00e-2 and the corresponding pointwise error $e=q_{h,\tau}^*-q^\dag$, plotted for $T=0.5$, for $\alpha=0.25$ (top), $\alpha=0.50$ (middle) and $\alpha=0.75$ (bottom). \label{fig:recon-exam3}}
\end{figure}

\section*{Acknowledgements}

The work of B. Jin was partially supported by UK EPSRC EP/T000864/1 and a start-up fund from The Chinese University of Hong Kong, and that of  Z. Zhou by Hong Kong Research Grants Council grant (Project No. 15304420) and an internal grant of Hong Kong Polytechnic University (Project ID: P0031041, Work Programme: ZZKS).

\appendix

\section{Proof of Theorem \ref{thm:weak-sol}}
\label{appA}

We give the proof of Theorem \ref{thm:weak-sol}. The argument follows largely that of
\cite[Theorem 6.14]{Jin:2021book}. It is provided only for completeness.

\begin{proof}
Let $w=u-u_0$. Then it suffices to show that there exists a unique solution $w \in L^p(0,T; W_0^{1,r}
\II)\cap C([0,T];L^r\II)$ and $\partial_t^\alpha w  \in L^p(0,T; W^{-1,r}\II)$, where $w$ satisfies
in $W^{-1,r}\II$
\begin{equation}\label{eqn:fde-w-1}
   \partial_t^\alpha w(t) + A(t) w(t) = f(t) - A(t)u_0,\quad t\in(0,T],\quad \text{with} ~~w(0)=0.
\end{equation}
For any $\theta\in[0,1]$, consider the auxiliary problem
\begin{align}\label{PDE-theta}
\begin{aligned}
&\Dal w(t) + A(\theta t)w(t)
= f(t) - A(t)u_0, ~~t\in(0,T],\quad \mbox{with }
w(0)=0,
\end{aligned} 
\end{align}
and let
$D=\{\theta\in[0,1]:\mbox{\eqref{PDE-theta} has a solution }
u\in L^p(0,T; W_0^{1,r}(\Omega)) \mbox{ such that } \Dal w\in L^p(0,T; W^{-1,r}(\Omega)) \}.$
Next we prove that the set $D$ is a closed subset of $[0,1]$.
Since $u_0\in W^{1,r}_0(\Omega)$, we have $f - A(t)u_0 \in L^p(0,T; W^{-1,r}\II)$,
and by Lemma \ref{lem:max-weak}, we deduce $0\in D$ and $D\neq \emptyset $.
Then for any $\theta\in D$ and $t_0 \in (0,T]$, we rewrite \eqref{PDE-theta} as
\begin{align*}
\begin{aligned}
&\Dal w(t) + A(\theta t_0)w(t) = f(t) - A(t)u_0 + (A(\theta t_0)-A(\theta t))w(t) ,\quad t\in(0,T],\quad \mbox{with }
w(0)=0.
\end{aligned}
\end{align*}
By the maximal $L^p$ regularity in Lemma \ref{lem:max-weak} and the perturbation estimate in Lemma \ref{lem:perturb}, we obtain
\begin{align}
&\|\Dal w\|_{L^p(0,t_0;W^{-1,r}(\Omega))}+\| \nabla w\|_{L^p(0,t_0;L^r(\Omega))} \nonumber\\
\le&c\|f- A(t)u_0\|_{L^p(0,t_0;W^{-1,r}(\Omega))}
+ c\|(A(\theta t_0)-A(\theta t))w(t)\|_{L^p(0,t_0; W^{-1,r}(\Omega))} \nonumber\\
\le&c\|f- A(t)u_0\|_{L^p(0,t_0;W^{-1,r}\II)}+ c\|(t_0-t) \nabla w\|_{L^p(0,t_0;L^r(\Omega))}.\label{maxLp-1}
\end{align}
Let $g(t)=\| \nabla w\|_{L^p(0,t; L^r(\Omega))}^p $. Since
$g'(t) = \| \nabla w(t)\|_{L^r(\Omega)}^p$ and $g(0)=0$, \eqref{maxLp-1} and integration by parts gives
\begin{align*}
g(t_0)&\le c\|f- A(t)u_0\|_{L^p(0,t_0;W^{-1,r}\II)}^p    + c\int_0^{t_0}(t_0-t)^p g'(t)\d t  \\
&= c\|f- A(t)u_0\|_{L^p(0,t_0;W^{-1,r}\II)}^p    + cp\int_0^{t_0}(t_0-t)^{p-1}g(t)\d t.
\end{align*}
Then the standard Gronwall's inequality implies
\begin{align*}
 \| \nabla w\|_{L^p(0,t_0; L^r(\Omega))} \leq c\|f - A(t)u_0\|_{L^p(0,t_0; W^{-1,r}(\Omega))}.
\end{align*}
This inequality and \eqref{maxLp-1} yield
\begin{align}\label{maxLp-2}
&\|\Dal w\|_{L^p(0,t_0; W^{-1,r}(\Omega))}+\| \nabla w\|_{L^p(0,t_0; L^r(\Omega))} \le
c\|f- A(t)u_0\|_{L^p(0,t_0;W^{-1,r}(\Omega))} .
\end{align}
Since this estimate is independent of $\theta\in D$, $D$ is a closed subset of $[0,1]$.
Next we show that $D$ is open with respect to
the subset topology of $[0,1]$.  In fact, for any $\theta_0\in D$ and $\theta\in[0,1]$  close to $\theta_0$,
we rewrite  problem \eqref{PDE-theta}  as
\begin{align*}
\begin{aligned}
&\Dal w(t) + A(\theta_0 t)w(t) +(A(\theta t)-A(\theta_0 t))w(t)
= f(t) -A(t)u_0,\quad t\in(0,T],\quad \mbox{with }
w(0)=0 ,
\end{aligned}
\end{align*}
which is equivalent to
\begin{align*}
&[1 + (\Dal + A(\theta_0 t))^{-1} (A(\theta t)-A(\theta_0 t))]w(t)
= (\Dal + A(\theta_0 t))^{-1} (f(t)-A(t)u_0).
\end{align*}
The estimate \eqref{maxLp-2} and Lemma \ref{lem:perturb} imply that for any $v\in W_0^{1,r}\II$
\begin{align*}
&\| (\Dal + A(\theta_0 t))^{-1} (A(\theta t)-A(\theta_0 t)) v\|_{L^p(0,T;W^{1,r}(\Omega))}\\
\le& c\|  (A(\theta t)-A(\theta_0 t)) v \|_{L^p(0,T;W^{-1,r}(\Omega))}
\le c|\theta-\theta_0| \| \nabla  v \|_{L^p(0,T; L^{r}\II)}.
\end{align*}
Thus for $\theta$ sufficiently close to $\theta_0$, the operator $1 + (\Dal + A(\theta_0 t))^{-1} (A(\theta t)-
A(\theta_0 t)) $ is invertible on $L^p(0,T; W_0^{1,r}(\Omega))$, which implies $\theta\in D$. Thus $D$
is open with respect to the subset topology of $[0,1]$.
Since $D$ is both closed and open respect to the subset topology of $[0,1]$, we deduce $D=[0,1]$.
In sum, problem \eqref{eqn:fde-w-1} has a solution $w$ such that $w\in L^p(0,T; W_0^{1,r}\II)$
and  $\partial_t^\alpha w \in L^p(0,T; W^{-1,r}\II)$. Since $W^{-1,r}\II$ is UMD \cite[Proposition 4.2.17]{Hytonen:2016}, and $w(0)=0$, we deduce $w\in W^{\alpha,p}(0,T;W^{-1,r}\II)$.
By interpolation between $W^{\alpha,p}(0,T;W^{-1,r}\II)$ and $L^p(0,T; W_0^{1,r}\II)$ \cite[Theorem 5.2]{Amann:2000},
we derive $u\in W^{\frac{\alpha}2, p}(0,T; L^r\II)$.
This, Sobolev embedding theorem and the condition $p>\frac2\alpha$ imply $w\in C([0,T]; L^r\II)$.
Similarly, if $r > d$ and  $p>\frac{2r}{\alpha(r- d)}$, interpolation \cite[Theorem 5.2]{Amann:2000}
and Sobolev embedding theorem imply that for $\frac{1}{\alpha p}<\theta<(\frac12-\frac{d}{2r})$ 
\begin{equation}\label{eqn:reg-L2Linf}
\begin{split}
  u &\in W^{\alpha,p}(0,T;W^{-1,r}\II) \cap L^p(0,T; W_0^{1,r}\II)\\
  &\hookrightarrow W^{\alpha\theta,p}(0,T; W^{1-2\theta,r}(\Omega)) \hookrightarrow L^\infty((0,T)\times \Omega).
\end{split}
\end{equation}
This completes the proof of the theorem.
\end{proof}

\section{Nonsmooth data estimates}  \label{ssec:techest}

In this appendix, we collect several nonsmooth data estimates for the numerical approximations of
 the direct problem \eqref{eqn:fde}, which are central for deriving
the basic estimates in Section \ref{ssec:basic}. First, we provide two useful results, i.e., error estimate and
maximal $\ell^p$ regularity, for the following fully discrete scheme for problem
\eqref{eqn:fde}: find $U_h^n(q^\dag)\in X_h$ satisfying $U_h^0=P_hu_0$ and
\begin{align}\label{eqn:fully-00}
  \bar \partial_\tau^\alpha U_h^n(q^\dag) + A_h(q^\dag(t_n)) U_h^n(q^\dag) = P_hf(t_n)=: f^n, \quad n=1,2,\ldots,N.
\end{align}
\begin{lemma}\label{lem:err-parabolic}
Let $q^\dag$ be the exact coefficient and $u\equiv u(q^\dag)$ the solution to problem \eqref{eqn:var},
and $\{U_h^n(q^\dag)\}$  the solution to problem \eqref{eqn:fully-00}. Then
under Assumption \ref{ass:data-2}, the following error estimate holds
\begin{equation*}
\| u(t_n)-U_h^n(q^\dag)\|_{L^2(\Omega)} \le c(\tau t_n^{\alpha-1}  + h^2),\quad n=1,\ldots,N.
\end{equation*}
\end{lemma}
\begin{proof}
The error estimate improves upon a known result from \cite{JinLiZhou:2019}, by
removing the log factor $\ell_h=|\ln h|$, under
Assumption \ref{ass:data-2}. It suffices to show that for $u_0 \equiv 0$ and
$f\in C^1([0,T]; L^2\II)$
\begin{equation}\label{eqn:semi-err}
\|  (u_h - u)(t)  \|_{L^2\II} \le c h^2, \quad  \forall t\in [0,T],
\end{equation}
where $u_h$ is the solution to the semidiscrete scheme:
\begin{equation}\label{eqn:semi}
\partial_t^\alpha u_h (t) + A_h(q^\dag(t)) u_h(t) = P_h f(t), \quad \forall t\in (0,T],\quad \text{with}~~u_h(0)= 0.
\end{equation}
For any $t_* \in (0,T]$,
let $A_{h*} = A_h(q^\dag(t_*))$ and $A_h(t)=A_h(q^\dag(t))$, and further we define the solution operators $F_{h*}$ and $E_{h*}(t)$ by
$$F_{h*}(t)=\frac{1}{2\pi {\rm i}}\int_{\Gamma_{\theta,\delta }}e^{zt} z^{\alpha-1} (z^\alpha+A_{h*} )^{-1}\, \d z\quad \mbox{and}\quad
E_{h*}(t):=\frac{1}{2\pi {\rm i}}\int_{\Gamma_{\theta,\delta}}e^{zt}  (z^\alpha+A_{h*} )^{-1}\, \d z.$$
Then the solution $u_h$ is given by
\begin{align*}
u_h(t) =   \int_0^t E_{h*}(t-s)\big(P_hf(s) + (A_{h*}-A_h(s))u_h(s)\big) \d s ,
\end{align*}
Let $e_h=P_hu-u_h$. Then by \eqref{eqn:sol-op-rep}, $e_h$ is given by
\begin{align}\label{FE-Err-expr2}
e_h(t)=&  \int_0^t (P_hE_*(t-s)-E_{h*}(t-s)P_h)f(s)  \d s\nonumber\\
&+ \int_0^t (P_hE_*(t-s)-E_{h*}(t-s)P_h)(A_*-A(s))u(s) \d s \nonumber \\
&+ \int_0^t E_{h*}(t-s)\big(P_h(A_*-A(s))u(s) -(A_{h*}-A_h(s))u_h(s)\big) \d s =: \sum_{i=1}^3{\rm I}_i(t).
\end{align}
The argument in \cite[Theorem 3.3]{JinLiZhou:2019} gives
\begin{align}
  \|{\rm I}_2(t_*)\|_{L^2(\Omega)} & \leq ch^2 \|f\|_{L^\infty(0,t_*;L^2(\Omega))},\label{23-1}\\
  \|{\rm I}_3(t_*)\|_{L^2(\Omega)}&\leq ch^2\|f\|_{L^\infty(0,t_*;L^2(\Omega))} + c\int_0^{t_*} \|e_h(s)\|_{L^2(\Omega)}\d s,\label{23-2}
\end{align}
Now we bound the term ${\rm I}_1$ in \eqref{FE-Err-expr2}. The identities $E_*(t) = -A_*^{-1}F_*'(t)$ and
$E_{h*}(t) = -A_{h*}^{-1}F_{h*}'(t)$ \cite[Lemma 6.1]{Jin:2021book} and integration by parts imply
\begin{equation*}
{\rm I}_1(t)= (P_h F_*(t) A_*^{-1}-F_{h*}(t) A_{h*}^{-1} P_h ) f(0)   - \int_0^t  (  P_h F_*(t-s) A_*^{-1}-F_{h*}(t-s) A_{h*}^{-1}P_h)f'(s) \,\d s.
\end{equation*}
For any $v\in L^2\II$, we derive
\begin{align*}
F_{h*}(t) A_{h*}^{-1} P_hv &= \frac{1}{2\pi {\rm i}}\int_{\Gamma_{\theta,\delta }}\!\!\!e^{zt} z^{\alpha-1} (z^\alpha+A_{h*} )^{-1}  A_{h*}^{-1} P_h v\, \d z
= \frac{1}{2\pi {\rm i}}\int_{\Gamma_{\theta,\delta }}\!\!\!e^{zt} z^{-1} \Big[ A_{h*}^{-1} - (z^\alpha+A_{h*} )^{-1} \Big] P_hv \, \d z.
\end{align*}
Similarly, $F_*(t) A_*^{-1}v$ can be represented as
\begin{equation*}
F_*(t) A_*^{-1}v = \frac{1}{2\pi {\rm i}}\int_{\Gamma_{\theta,\delta }}e^{zt} z^{-1} \Big[ A_{*}^{-1} - (z^\alpha+A_{*} )^{-1} \Big]v \, \d z.
\end{equation*}
Then the standard finite element approximation yields that  for any $z\in \Gamma_{\theta,\delta }$ \cite[p. 819--820]{FujitaSuzuki:1991}
$\| ( A_{h*}^{-1} P_h  - A_{*}^{-1})v\|_{L^2\II}\le c h^2 \| v\|_{L^2\II}$ and
$\| (( z^\alpha + A_{h*})^{-1} P_h  - (z^\alpha+A_{*})^{-1})v\|_{L^2\II}\le c h^2 \|v \|_{L^2\II}$.
Consequently, we arrive at
$$  \|(P_h F_*(t) A_*^{-1}-F_{h*}(t) A_{h*}^{-1} P_h) v \|_{L^2\II}\le c h^2 \|v\|_{L^2\II},\quad \forall v\in L^2\II.$$
Therefore, we obtain
$$\|{\rm I}_1(t)\|_{L^2\II}\le ch^2 \| f(0) \|_{L^2\II} + ch^2\int_0^t \|  f'(s) \|_{L^2(\Omega)}\,\d s,$$
which together with \eqref{23-1}--\eqref{23-2} implies
$$ \|  e_h(t_*) \|_{L^2\II} \le c h^2 + \int_0^{t_*} \|e_h(s)\|_{L^2(\Omega)}\d s.$$
The desired result follows from Gronwall's inequality and the triangle inequality.
\end{proof}

The next result gives the maximal $\ell^p$ regularity for the scheme \eqref{eqn:fully-00}, where $A_h$
denotes the discrete negative Dirichlet Laplacian.
\begin{lemma}\label{lem:disc-max-reg}
Let $\{U_h^n\}_{n=1}^N$ be the solution to the scheme \eqref{eqn:fully-00} with $u_0\equiv 0$. Then for any $p\in(1,\infty)$,
$$ \| (\bar\partial_\tau^\alpha U_h^n)_{n=1}^N  \|_{\ell^p(L^2\II)}
+  \| (A_h U_h^n)_{n=1}^N  \|_{\ell^p(L^2\II)} \le c \|  (f^n)_{n=1}^N  \| _{\ell^p(L^2\II)}.$$
\end{lemma}
\begin{proof}
For any $m=1,2,\ldots,N$, the scheme \eqref{eqn:fully-00} can be recast into
\begin{align*}
\bar\partial_\tau^\alpha U_h^n + A_h(q^\dag(t_m)) U_h^n
= P_hf^n+  (A_h(q^\dag(t_m))-A_h(q^\dag(t_n)) ) U_h^n. 
\end{align*}
Since $A_h(q^\dag(t_m))$ is independent of $n$, there holds the discrete maximal $\ell^p$ regularity \cite{JinLiZhou:2018nm}
\begin{align*}
\begin{aligned}
 &\quad \| (\bar\partial_\tau^\alpha U_h^n)_{n=1}^m  \|_{\ell^p(L^2\II)}^p
 + \| (A_h U_h^n)_{n=1}^m  \|_{\ell^p(L^2\II)} ^p \\
&\le c \big(\|  (f^n)_{n=1}^m  \| _{\ell^p(L^2\II)}^p 
+  \| [(A_h(q^\dag(t_m))-A_h(q^\dag(t_n)) ) U_h^n  ]_{n=1}^m \|_{L^2\II}^p\big). 
\end{aligned}
\end{align*}
Note that under condition \eqref{Cond-q}, there holds \cite[Remark 3.1]{JinLiZhou:2019}
\begin{equation*}
  \| (A_h(t) - A_h(s))v_h\| \le c\, |t-s|\, \|  A_hv_h \|,\quad \forall v_h \in X_h.
\end{equation*}
Consequently,
\begin{align*}
\begin{aligned}
&\quad \| (\bar\partial_\tau^\alpha U_h^n)_{n=1}^m  \|_{\ell^p(L^2\II)}^p
 + \| (A_h U_h^n)_{n=1}^m  \|_{\ell^p(L^2\II)} ^p\\
&\le c\|  (f^n)_{n=1}^m  \| _{\ell^p(L^2\II)}^p  
+ c \tau \sum_{n=1}^m |t_m - t_n|^p \| A_h U_h^n   \|_{L^2\II}^p.
\end{aligned}
\end{align*}
Let $g^m = \| (A_h U_h^n)_{n=1}^m  \|_{\ell^p(L^2\II)} ^p$. 
  Then the above estimate implies
\begin{align*}
  g^m &\le  c \|  (f^n)_{n=1}^m  \| _{\ell^p(L^2\II)}^p  
+ c \tau \sum_{n=1}^m |t_m - t_n|^p  \frac{g^n - g^{n-1}}{\tau} \\
&\le  c\|  (f^n)_{n=1}^m  \| _{\ell^p(L^2\II)}^p + c \tau \sum_{n=1}^{m-1} \frac{(t_m - t_n)^p-(t_m - t_{n+1})^p}{\tau}  g^n  \\
&\le c\|  (f^n)_{n=1}^m  \| _{\ell^p(L^2\II)}^p
+ c \tau  \sum_{n=1}^{m-1} t_{m-n}^{p-1} g^n  .
\end{align*}
Then the standard discrete Gronwall's inequality leads to
\begin{align*}
\begin{aligned}
 \| (A_h U_h^n)_{n=1}^m  \|_{\ell^p(L^2\II)} ^p
\le c \|  (f^n)_{n=1}^m  \| _{\ell^p(L^2\II)}^p .
\end{aligned}
\end{align*}
and the desired result follows immediately by the triangle's inequality.
\end{proof}

The next lemma provides an error estimate of the scheme \eqref{eqn:fully-00} with
the (perturbed) coefficient $\mathcal{I}_hq^\dag$. 
\begin{lemma}\label{lem:err-1}
Let $q^\dag$ be the exact diffusion coefficient, $u\equiv u(q^\dag)$ the solution to problem \eqref{eqn:var},
and  $\{U_h^n(\mathcal{I}_hq^\dag)\}\subset X_h$ the numerical solutions to the scheme
\eqref{eqn:fully-00} with $\mathcal{I}_hq^\dag$ in place of $q^\dag$. Then
under Assumptions \ref{ass:data-2} and \ref{ass:zdelta}, 
\begin{equation*}
\begin{split}
\|(u(t_n)-U_h^n(\mathcal{I}_h q^\dag))_{n=1}^N\|_{\ell^2(L^2(\Omega))}^2 &\le  \begin{cases}
 c (\tau^{\min(2,1+2\alpha)} +h^4),&\alpha\neq1/2;\\
 c (\tau^2\ell_N + h^4),&\alpha=1/2.
\end{cases}
\end{split}
\end{equation*}
\end{lemma}
\begin{proof}
Note that $U_h^n(q^\dag)$ and  $U_h^n(\mathcal{I}_h q^\dag)$ satisfy $U_h^0(q^\dag) = U_h^0(\mathcal{I}_hq^\dag) = P_h u_0$ and
\begin{align*}
    \bar \partial_\tau^\alpha U_h^n(q^\dag) +   A_h(q^\dag(t_n)) U_h^n(q^\dag) &=  P_h f(t_n),\quad n=1,2\ldots,N,\\
 \bar \partial_\tau^\alpha  U_h^n(\mathcal{I}_h  q^\dag) + A_h(\mathcal{I}_h q^\dag(t_n))   U_h^n(\mathcal{I}_h q^\dag) &=  P_h f(t_n),\quad n=1,2,\ldots, N.
\end{align*}
By subtracting the two identities, we deduce that
$\rho_h^n:=U_h^n(q^\dag)  - U_h^n(\mathcal{I}_h q^\dag)$ satisfies $\rho_h^0=0$ and
\begin{equation}\label{eqn:err-eq-1}
   \bar \partial_\tau^\alpha \rho_h^n + A_h(q^\dag(t_n))\rho_h^n = \big( A_h(\mathcal{I}_h q^\dag(t_n))  - A_h(q^\dag(t_n))\big) U_h^n(\mathcal{I}_h  q^\dag), \quad n=1,\ldots,N.
\end{equation}
The the maximal $\ell^p$ regularity in Lemma \ref{lem:disc-max-reg} implies
\begin{align*}
 \|(\rho_h^n)_{n=1}^N\|_{\ell^2(L^2(\Omega))}^2
&\le c    \|(A_h(q^\dag(t_n))^{-1}\big(A_h(\mathcal{I}_h q^\dag(t_n))  - A_h(q^\dag(t_n))\big) U_h^n(\mathcal{I}_h  q^\dag))_{n=1}^N\|_{\ell^2(L^2\II)}^2\\
&\le  c  \| (\big(A_h(\mathcal{I}_h q^\dag(t_n))^{-1}  - A_h(q^\dag(t_n))^{-1}\big) A_h(\mathcal{I}_h q^\dag(t_n)) U_h^n(\mathcal{I}_h  q^\dag))_{n=1}^N\|_{\ell^2(L^2\II)}^2.
\end{align*}
By \cite[Lemma A.1]{JinZhou:2021sinum}, we have for any $\epsilon>0$ and $p\ge \max(d+\epsilon,2)$,
\begin{equation*}
   \| A_h(\mathcal{I}_h q^\dag)^{-1} - A_h(q^\dag)^{-1}  \|_{L^p(\Omega)\rightarrow L^2(\Omega)} \le c h^2.
\end{equation*}
Consequently,
\begin{equation*}
 \|(\rho_h^n)_{n=1}^N\|_{\ell^2(L^2(\Omega))}^2
\le  c h^4 \|  (A_h(\mathcal{I}_h q^\dag(t_n)) U_h^n(\mathcal{I}_h  q^\dag))_{n=1}^N\|_{\ell^2(L^p(\Omega))}^2.
\end{equation*}
Then the maximal $\ell^p$ regularity for the backward Euler CQ in Lemma \ref{lem:disc-max-reg} implies
\begin{equation*}
    \|  (A_h(\mathcal{I}_h q^\dag(t_n)) U_h^n(\mathcal{I}_h  q^\dag))_{n=1}^N\|_{\ell^2(L^p\II)}^2\le
  c (\|( f(t_n))_{n=1}^N\|_{\ell^2(L^p(\Omega))}^2  +  \| \nabla u_0 \|_{L^p(\Omega)} ^2).
\end{equation*}
Finally, the desired estimate follows from Lemma \ref{lem:err-parabolic} and the  triangle inequality.
\end{proof}

Last, we give an estimate on the backward Euler CQ approximation of $\partial_t^\alpha u(t_n)$.
\begin{lemma}\label{lem:deriv-approx}
Let $q^\dag$ be the exact diffusion coefficient and $u\equiv u(q^\dag)$ be the solution to problem \eqref{eqn:var}.
Then under Assumption \ref{ass:data-2}, with $\ell_n = \ln(1+\frac{t_n}{\tau})=\ln(n+1)$,
 there holds
\begin{equation*}
  \| \bar \partial_\tau^\alpha u(t_n) - \partial_t^\alpha u(t_n)   \|_{L^2(\Omega)}
           \le c\tau ( t_n^{-1} + \ell_n).
\end{equation*}
\end{lemma}
\begin{proof}
The proof employs a (different) perturbation argument. Let $A_0=A(0)$.
Let $F(t)=(A_0-A(t))u(t)+f(t)$ and $y(t)=u(t)-u_0$. Then $y(t)$ satisfies
\begin{equation*}
  \partial_t^\alpha y(t) + A_0y(t) = F(t) - A_0u_0,  \quad \forall t\in(0,T],\quad \text{with}~~ y(0)=0.
\end{equation*}
Using the identity $F(t) = F(0)+ \int_0^t F'(s)\,\d s$, then Laplace transform gives
$$z^\alpha \widehat y(z) + A_0\widehat y(z)   =   z^{-1}(F(0) - A_0u_0) + z^{-1} \widehat{F'}(z),$$
i.e.,
$$\widehat y(z) = (z^\alpha+A_0)^{-1} ( z^{-1}(F(0) - A_0u_0) + z^{-1} \widehat{F'}(z)).$$
Similarly, one can derive a representation for the discrete approximation.
By inverse Laplace transform, $w^n=\partial_t^\alpha y(t_n)-\bar \partial_\tau^\alpha y(t_n)$ is given by
\begin{align*}
 w^n  & =  \frac{1}{2\pi\mathrm{i}} \int_{\Gamma_{\theta,\delta}^\tau} e^{zt_n} K(z) (z^{-1}(F(0) - A_0u_0) + z^{-1} \widehat{F'}(z))\,\d z \\
 &\quad + \frac{1}{2\pi\mathrm{i}} \int_{\Gamma_{\theta,\delta}\setminus\Gamma_{\theta,\delta}^\tau} e^{zt_n} K(z)(z^{-1}(F(0) - A_0u_0) + z^{-1} \widehat{F'}(z))\,\d z.
\end{align*}
with $\Gamma_{\theta,\delta}^\tau=\{z \in \Gamma_{\theta,\delta}, |\rm{Im}(z)|\le \frac\pi\tau \}$ and
$$K(z)= (z^\alpha-\delta_\tau(e^{-z\tau})^{\alpha}) (z^\alpha+A_0)^{-1},$$
{with $\delta_\tau(\xi)=\tau^{-1}(1-\xi)$ being characteristic polynomial
of the backward Euler method.} Simple computation shows that the following estimates hold
\begin{align}
  \quad c_1 |z| &\le |\delta_\tau(e^{-z\tau})| \le c_2|z|,\qquad  | \delta_\tau(e^{-z\tau})^{\alpha}-z^\alpha| \le c \tau z^{1+\alpha}, \quad \forall z\in \Gamma_{\theta,\delta}^\tau,\label{eqn:gen}\\
  |\delta_\tau(e^{-z\tau})| &\le |z| \sum_{k=1}^\infty \frac{|z\tau|^{k-1}}{k!} \leq |z|e^{|z|\tau}, \quad \forall z\in \Sigma_\theta=\{z\in\mathbb{C}: z\neq0, |\arg(z)|\leq \theta\},\label{eqn:est-kernel}
\end{align}
and the resolvent estimate
\begin{equation}\label{eqn:resol}
  \|(z+A_0)^{-1}\|\leq c|z|^{-1},\quad \forall z\in \Sigma_\theta.
\end{equation}
We first treat the error involving $(A_0u_0-F(0))$, and let
\begin{equation*}
  {\rm I}_1=\frac{1}{2\pi\mathrm{i}} \int_{\Gamma_{\theta,\delta}^\tau} e^{zt_n} K(z) z^{-1}(A_0u_0-F(0))\d z
  \quad\mbox{and}\quad
  {\rm I}_2=\frac{1}{2\pi\mathrm{i}} \int_{\Gamma_{\theta,\delta}\setminus\Gamma_{\theta,\delta}^\tau} e^{zt_n} K(z)z^{-1}(F(0)-A_0u_0)\d z.
\end{equation*}
By choosing $\delta=c/t_n$ in $\Gamma_{\theta,\delta}$ and applying \eqref{eqn:resol},
the term ${\rm I}_1$ is bounded by
\begin{align*}
\| {\rm I}_1\|_{L^2(\Omega)}
   &\le c \tau \| F(0) - A_0 u_0 \|_{L^2(\Omega)} \Big(\int_{\frac{c}{t_n}}^{\frac{\pi\sin\theta}{\tau}} e^{-c\rho t_n} \,\d\rho
   + \int_{-\theta}^\theta ct_n^{-1} \,\d\theta \Big)
  \le c\tau t_n^{-1}  \| F(0) - A_0 u_0\|_{L^2(\Omega)}.
\end{align*}
Further, by \eqref{eqn:est-kernel}, for any $z=\rho e^{\pm\mathrm{i}\theta}\in \Gamma_{\theta,\delta}\setminus\Gamma_{\theta,\delta}^\tau$ and
choosing $\theta\in(\pi/2,\pi)$ close to $\pi$,
\begin{align*}
  |e^{zt_n} (\delta_\tau(e^{-z\tau})^{\alpha}-z^\alpha)z^{ -1}| & \leq e^{t_n\rho\cos \theta}(c|z|^\alpha e^{\alpha\rho\tau}+|z|^\alpha)|z|^{-1}\leq c|z|^{\alpha-1}e^{-c\rho t_n}.
\end{align*}
Then the term ${\rm I_2}$ is bounded by
\begin{equation*}
\|{\rm I}_2\|_{L^2(\Omega)}
   \le  c \| F(0)-A_0u_0\|_{L^2(\Omega)}  \int_{\frac{\pi\sin\theta}{\tau}}^\infty e^{-c\rho t_n} \rho^{-1}\,\d\rho
    \le c\tau t_n^{-1} \| F(0)-A_0u_0\|_{L^2(\Omega)}.
\end{equation*}
This argument also bounds for the term involving $\widehat{F'}(z)$. Finally, we obtain
 \begin{equation*}
\begin{split}
\|w^n\|_{L^2(\Omega)}
   &\le c\tau t_n^{-1} \| F(0)-A_0u_0\|_{L^2(\Omega)}+ \int_\tau^{t_{n}} (t_n-s+\tau)^{-1} \|  F'(s)  \|_{L^2\II}\,\d s.
\end{split}
\end{equation*}
Then the solution regularity \eqref{reg-fde-2} and the perturbation estimate \eqref{eqn:perturb} immediately imply
\begin{align*}
   \|  F'(s)  \|_{L^2\II} &\leq \|  f'(s)  \|_{L^2\II} + \| A'(s) u(s)  \|_{L^2\II} + \|  (A_0 - A(s)) u'(s) \|_{L^2\II}\\
      &\leq  c(\|  f'(s)  \|_{L^2\II} + \| u(s)  \|_{H^2\II} + s\| u'(s) \|_{H^2\II})\leq c.
\end{align*}
This bound and the estimate $\| f(0)-A_0u_0\|_{L^2(\Omega)}\leq c$ imply
\begin{align*}
\|w^n\|_{L^2(\Omega)}
  \le c\tau t_n^{-1}  + c\int_\tau^{t_{n}} (t_{n+1}-s )^{-1}  \,\d s \le c\tau (t_n^{-1} + \ell_n).
\end{align*}
This completes the proof of the lemma.
\end{proof}

\label{appB}

\bibliographystyle{abbrv}
\bibliography{frac_inv}

\end{document}